\documentclass[12pt]{article}
\usepackage{amssymb}
\usepackage{amsmath,bm,natbib,graphics,mathrsfs,graphicx,epsfig,subfigure,placeins,indentfirst,setspace,enumerate,enumitem,caption,varioref,booktabs,tabularx,color,epstopdf,multirow}
\setcitestyle{authoryear,open={(},close={)}}
\usepackage{amsthm}
\makeatletter \oddsidemargin  -.1in \evensidemargin -.1in
\begin{document}
	\newcommand{\bea}{\begin{eqnarray}}
		\newcommand{\eea}{\end{eqnarray}}
	\newcommand{\nn}{\nonumber}
	\newcommand{\bee}{\begin{eqnarray*}}
		\newcommand{\eee}{\end{eqnarray*}}
	\newcommand{\lb}{\label}
	\newcommand{\nii}{\noindent}
	\newcommand{\ii}{\indent}
	\newtheorem{theorem}{Theorem}[section]
	\newtheorem{example}{Example}[section]
	\newtheorem{corollary}{Corollary}[section]
	\newtheorem{definition}{Definition}[section]
	\newtheorem{lemma}{Lemma}[section]
	\newtheorem{remark}{Remark}[section]
	\newtheorem{proposition}{Proposition}[section]
	\numberwithin{equation}{section}
	\renewcommand{\qedsymbol}{\rule{0.7em}{0.7em}}
	\renewcommand{\theequation}{\thesection.\arabic{equation}}
	\renewcommand\bibfont{\fontsize{10}{12}\selectfont}
	\setlength{\bibsep}{0.0pt}
		\title{\bf Extended fractional cumulative past and paired $\phi$-entropy measures**}
	
\author{ Shital {\bf Saha}\thanks {Email address: shitalmath@gmail.com} ~and  Suchandan {\bf  Kayal}\thanks {Email address ~(corresponding author):
		~kayals@nitrkl.ac.in,~~suchandan.kayal@gmail.com
		\newline**It has been accepted on \textbf{Physica A: Statistical Mechanics and its Applications}.}
	\\{\it \small Department of Mathematics, National Institute of
		Technology Rourkela, Rourkela-769008, India}}
		
\date{}
\maketitle
		\begin{center}
Abstract
		\end{center}
			Very recently, extended fractional cumulative residual entropy (EFCRE) has been proposed by \cite{foroghi2022extensions}.  In this paper, we introduce extended fractional cumulative past entropy (EFCPE), which is a dual of the EFCRE. The newly proposed measure depends on the logarithm of fractional order and the cumulative distribution function (CDF). Various properties of the EFCPE have been explored. This measure has been extended to the bivariate setup. Furthermore, the conditional EFCPE is studied and some of its properties are provided. The EFCPE for inactivity time has been proposed. In addition, the extended fractional cumulative paired $\phi$-entropy has been introduced and studied. The proposed EFCPE has been estimated using empirical CDF. Furthermore, the EFCPE is studied for coherent systems. A validation of the proposed measure is provided using logistic map. Finally, an application is reported.
		 \\
			\textbf{MSCs:} 94A17; 60E15; 62B10.
			
			\section{Introduction}
		The concept of entropy is applied to measure the disorder or randomness associated with a system. It depends on randomness of the system's states. The entropy of a system with certain number of states is maximum when the random states have equal probability. Its value is zero (minimum) for a specific certain state of the system. 
	   The notion of entropy was first proposed by \cite{shannon1948mathematical}. Shannon developed the fundamental laws of data compression and transmission, which result the birth of the modern information theory. Later, \cite{jaynes1957information} proposed principle of maximum entropy, employed by several researchers in different areas such as environmental engineering, water resources and hydrology. \cite{renyi1961measures} proposed a one-parameter generalization of the Shannon entropy, which is additive in nature. \cite{tsallis1988possible} developed a non-additive entropy, another one-parameter generalization of the Shannon entropy. 
	   
	   Apart from these entropies, other types of generalizations of the Shannon entropy have been proposed by several researchers. For example, \cite{wang2003extensive} introduced incomplete extensive fractional entropy and applied it to study the correlated electron systems in weak coupling regime. Later, \cite{ubriaco2009entropies} proposed fractional entropy using the concept of fractional calculus. The author showed that the fractional entropy and Shannon entropy share similar properties except additivity. Let a discrete type random variable $X$ take values $x_{i}$ with probabilities $p_{i}$, $i=1,\ldots,n.$ Then, the fractional entropy of $X$ is given by
	   \begin{eqnarray}\label{eq1.1}
	   S_{q}(X)=\sum_{i=1}^{n}p_{i}(-\log p_{i})^{q},~~0\le q \le 1.
	   \end{eqnarray} 
	   It is shown by \cite{ubriaco2009entropies}  that the fractional entropy satisfies the Lesche and thermodynamic stability criteria. Eq. (\ref{eq1.1}) reduces to the Shannon entropy when $q$ becomes $1.$ The fractional entropy takes positive values. Further, it is concave and nonadditive in nature. 
	   
	   The concept of entropy for discrete type random variable can be written in the continuous domain. For continuous case, the Shannon entropy is known as differential entropy. Let $X$ be a non-negative and absolutely continuous random variable with probability density function (PDF) $f(.)$. The fractional (differential) entropy of $X$ is given by
	   \begin{eqnarray}\label{eq1.2}
	   H_{q}(X)=\int_{0}^{\infty}f(x)(-\log f(x))^{q}dx,~~0\le q \le 1.
	   \end{eqnarray}
	   Recently, motivated by cumulative residual entropy (see \cite{rao2004cumulative}), \cite{xiong2019fractional} introduced the concept of fractional cumulative residual entropy of $X$, which is given by
	   \begin{eqnarray}\label{eq1.3}
	   \mathcal{E}_{q}(X)=\int_{0}^{\infty}\bar{F}(x)(-\log \bar{F}(x))^{q}dx,~~0\le q\le 1,
	   \end{eqnarray} 
	   where $\bar{F}(.)$ is the survival function of $X$. The authors substituted the survival function in place of PDF in (\ref{eq1.2}) to get (\ref{eq1.3}). A fractional generalized cumulative residual entropy was proposed and studied by \cite{di2021fractional} (see Eq. (6)). \cite{tahmasebi2021results} applied fractional cumulative residual entropy for the coherent system lifetimes having identically distributed components.  Very recently, \cite{kayid2022some} considered fractional cumulative residual entropy and explored some further properties of it. 
	   
	    Mittag-Leffler function (see \cite{mittag1903nouvelle}) arises naturally in the solution of fractional order differential equations or fractional order integral equations. Particularly, it appears in the investigations of the fractional generalization of the kinetic equation, super-diffusion transport and in the study of several complex systems. The Mittag-Leffler function (MLF) is defined as
	   	\begin{eqnarray}
	   		E_{\alpha}(x)=\sum_{k=0}^{\infty}\frac{x^{k}}{(\alpha k)!},~0<\alpha<1,
	   	\end{eqnarray}
   	where $(\alpha k)!=\Gamma(\alpha k+1)$ and $\Gamma(.)$ is complete gamma function. It can be established that the inverse of the MLF is the solution of the functional equation 
   	\begin{eqnarray}
   		f(xy)=f(x)+f(y),~~x,y>0,
   	\end{eqnarray} 
   where $f(.):\mathcal{R}\rightarrow \mathcal{R}$ is a real-valued continuous function, which is not differentiable but has only derivative of order $\alpha,~0<\alpha<1.$ The inverse of MLF is also known as the fractional order logarithmic function, denoted by $Ln_{\alpha}x.$ Some of the important properties of $Ln_{\alpha}x$, for $0<\alpha<1$ are provided below (see \cite{jumarie2012derivation}).
   \begin{itemize}
   	\item $Ln_\alpha1=0,~ Ln_\alpha0=-\infty$, $Ln_\alpha x<0$, for $x<1$;
   	\item $1(Ln_\alpha1)^\frac{1}{\alpha}=0=0(Ln_\alpha0)^\frac{1}{\alpha}$;
   	\item 	$[Ln_{\alpha}uv]^{\frac{1}{\alpha}}=[Ln_{\alpha}u]^{\frac{1}{\alpha}}+[Ln_{\alpha}v]^{\frac{1}{\alpha}}$;
   	\item $Ln_{\alpha}(x^b)=b^\alpha Ln_{\alpha}x;$
   	\item $\frac{d^{\alpha}(Ln_{\alpha}x)^{\frac{1}{\alpha}}}{dx^{\alpha}}=\frac{\alpha!}{((1-\alpha)!)^{2}}\frac{1}{x^\alpha}.$
   \end{itemize}
Recently, the complexity of ultraslow diffusion process has been studied by \cite{liang2018diffusion} using both classical Shannon entropy and its general case with inverse MLF in conjunction with the structural derivative. The author has observed that the inverse Mittag-Leffler tail in
the propagator of the ultraslow diffusion equation model adds more information
to the original distribution with larger entropy. Further, the smaller value of
$\alpha$ in the inverse MLF indicates more complicated of the
underlying ultraslow diffusion and corresponds to higher value of entropy. As a result, the proposed definition of fractional entropy based on inverse MLF can be considered as an alternative measure to
capture the information loss in ultraslow diffusion. For details, please refer to \cite{liang2018diffusion}.
Based on the concept of inverse MLF, \cite{jumarie2012derivation} proposed a fractional entropy of order $\alpha$ of a discrete type random variable $X$ as 
	   	\begin{eqnarray}\label{eq1.4}
	   \bar {H}_\alpha(X)=-\sum_{i=1}^{n}p_i(Ln_\alpha p_i)^\frac{1}{\alpha},~~0<\alpha<1.
	   \end{eqnarray}
	   Note that the fractional entropy given by (\ref{eq1.4}) may be negative. For example, if we assume $\alpha=1/2,$ clearly, $\bar {H}_\alpha(X)$ takes negative values. Due to this reason, \cite{zhang2021cumulative} proposed another form of the fractional entropy of order $\alpha$ as  
	   	   	\begin{eqnarray}\label{eq1.4*}
	   	H_\alpha(X)=\sum_{i=1}^{n}p_i(-Ln_\alpha p_i)^\frac{1}{\alpha},~~0<\alpha<1,
	   \end{eqnarray}   
	   which is always non-negative.   Motivated by the fractional entropy given by (\ref{eq1.4*}),  and the fractional cumulative residual entropy given by (\ref{eq1.3}), \cite{foroghi2022extensions} proposed extended version of the fractional cumulative residual entropy, which is given by 
	   	\begin{eqnarray} \label{eq1.5}
	   \mathcal{E}_\alpha(X)=\int_{0}^{\infty} \bar F(x)[-Ln_\alpha \bar F(x)]^\frac{1}{\alpha} dx,~~0<\alpha< 1.
	   \end{eqnarray}
	   The authors studied bivariate version of (\ref{eq1.5}). Some bounds and stochastic ordering results are also explored by \cite{foroghi2022extensions}. We note that parallel to the fractional cumulative residual entropy, the concept of fractional cumulative (past) entropy has been developed and studied by \cite{di2021fractional}.  Recently, more results for the fractional cumulative entropy have been obtained by \cite{kayid2022some}. 
	   
	   On the basis of the aforementioned findings, in this communication, we propose an extended version of the fractional cumulative past entropy. The newly proposed measure has been defined in the next section similar to (\ref{eq1.5}). In order to get the newly proposed measure, we replace the survival function by cumulative distribution function (CDF) $F(.)$ of $X$ in (\ref{eq1.5}). We remark that in recent years, there have been various attempts to introduce fractional versions of the uncertainty measures. The fractional versions of entropies have found some applications in the areas related to complex systems, where the classical Shannon entropy has some limitations. Various important properties of the fractional calculus allow the fractional uncertainty measures to capture long-range phenomena, higher sensitivity in signal evolution and nonlocal dependence in some random systems in better way. For example, \cite{zhang2019uncertainty} used discrete fractional cumulative residual entropy to analyze the financial time-series data.  For time-series data, \cite{lopes2020review} computed values of various  fractional uncertainty measures.  \cite{wang2020complexity} proposed generalized fractional cumulative residual distribution entropy and showed that it can capture the tiny evolution of signal data better than generalized cumulative residual distribution entropy.
	   
	   The rest of the paper is organized as follows. In the next section, we introduce the concept of EFCPE and studied various properties of it. Bivariate EFCPE is proposed and its properties are studied. Some bounds are obtained. Further, we propose conditional EFCPE and dynamic version of the EFCPE. In Section $3,$ the concept of extended fractional paired $\phi$-entropy has been explored. The stability of the EFCPE has been discussed.  Empirical EFCPE is studied in Section $4.$ The proposed measure has been studied for coherent systems in Section $5$. Validation of the proposed measure using simulation on logistic map is provided in Section $6.$ An application is also explained.  Section $7$ concludes the paper.

	   \section{Extended fractional cumulative past entropy} \setcounter{equation}{0}
	   In this section, we define EFCPE and discuss its various properties. The newly proposed measure is useful to quantify information for the inactivity time of a system. The inactivity time is the time elapsing between the failure of a system and the time when it is found to be down. In other terms, our information measure, that will be called ``EFCPE" is suitable to measure information when uncertainty is related
	   	to the past. Further, note that the EFCPE is  dual of the EFCRE and it is well known that the EFCRE measures information when the uncertainty is related to future. Throughout the paper, we assume that the random variables are non-negative and absolutely continuous.
	   	\begin{definition}\label{def1.1}
	   	Suppose $X$ is a non-negative absolutely continuous random variable with CDF $F(.)$ and PDF $f(.)$. Then, the EFCPE is defined as
	   	\begin{eqnarray} \label{eq2.1}
	   	\mathcal{E}^*_\alpha(X)~ (or~ \mathcal{E}^*_\alpha(F))=\int_{0}^{\infty}  F(x)[-\text{Ln}_\alpha F(x)]^\frac{1}{\alpha} dx=E\bigg(\frac{[-\text{Ln}_\alpha  F(X)]^\frac{1}{\alpha}}{r(X)}\bigg),~~0<\alpha< 1,
	   	\end{eqnarray}
	   	where $r(.)=f(.)/F(.)$ denotes the reversed hazard rate of $X$. 
	   	 \end{definition}
   	 \noindent Next, we present some basic properties of the EFCPE. We recall that similar properties also hold for the measures proposed by \cite{xiong2019fractional}, \cite{di2021fractional} and \cite{foroghi2022extensions}.
   	 
   	 \begin{itemize}
   	 \item In (\ref{eq2.1}), the argument of the fractional order logarithmic function is $F(x)=P(X\le x)$, which guarantees that  $\mathcal{E}^*_\alpha(X)\ge0$. Indeed,  $0\le \mathcal{E}^*_\alpha(X)\le \infty$. For a degenerate random variable $X$, $\mathcal{E}^*_\alpha(X)=0$. 
   	 \item Let $X$ be a symmetric random variable with CDF $F(.)$ and finite mean $m=E(X)$. Then, $F(m+x)=1-F(m-x),$ for all $x\in\mathbb{R}.$ Thus, clearly, $\mathcal{E}_\alpha(X)=\mathcal{E}^*_\alpha(X).$
   	 \item Suppose that $X$ is a non-negative absolutely continuous random variable with CDF $F(.)$ and $Y=aX+b$, where $a>0$ and $b\geq0$. Then, we have 
   	 	$\mathcal{E}^*_\alpha(Y)= a \mathcal{E}^*_\alpha(X)$, which implies that the newly proposed measure is shift-independent. 
   	 \end{itemize}
   	 
   	 Utilizing the relation $\text{Ln}_\alpha p\approx\log p^{\alpha!},$ $0<\alpha<1,$ (see p. 125 of \cite{jumarie2012derivation})  where $\alpha!=\Gamma(1+\alpha)$ and $\Gamma(.)$ is a complete gamma function in (\ref{eq2.1}), we get an approximation for EFCPE, which is given by
	   	\begin{eqnarray}\label{eq2.2}
	   	\mathcal{E}^*_\alpha(X)\approx (\alpha!)^\frac{1}{\alpha} \int_{0}^{\infty}  F(x)[-\log F(x)]^\frac{1}{\alpha} dx=(\alpha!)^\frac{1}{\alpha}E\bigg(\frac{[-\log F(X)]^\frac{1}{\alpha}}{r(X)}\bigg),~~0<\alpha< 1.
	   	\end{eqnarray}
	   	\cite{foroghi2022extensions} proposed a special type of EFCRE based on the concept of fractional order logarithmic function. Similarly, herein, we propose a modified EFCPE, which is given by
	   		\begin{eqnarray} \label{eq2.3}
	   	\mathcal{\bar E}^*_\alpha(X)=\int_{0}^{\infty}  F(x) [-\text{Ln}_\alpha F(x)]dx\approx -\alpha!\int_{0}^{\infty} F(x) \log F(x) dx	=\alpha! \mathcal{E}^*(X),~0<\alpha< 1,
	   	\end{eqnarray}
	   	where $\mathcal{E}^*(X)$ is known as the cumulative entropy (see Eq. (3) of  \cite{di2009cumulative}). Now, we obtain EFCPE for some well-known distributions, say uniform and Fr\'echet distributions. Denote by $\Gamma(.)$ the complete gamma function. 
	   	\begin{example}\label{ex2.1}~~
	   		\begin{itemize}
	   		\item[(i)] Let $X$ follow uniform distribution in the interval $[0,a]$ with CDF $F(x)=x/a,~0<x<a.$ Then, $\mathcal{E}^*_\alpha(X)\approx \frac{a(\alpha!)^{1/\alpha}\Gamma(\frac{1}{\alpha}+1)}{2^{\frac{1}{\alpha}+1}},~0<\alpha<1.$
	   		\item[(ii)] Let $X$ follow Fr\'echet distribution with CDF $F(x)=e^{-bx^{-a}},~x>0,~a,~b>0.$ Then, 
	   		$\mathcal{E}^*_\alpha(X)\approx \frac{(\alpha!)^{1/\alpha}b^{1/a}\Gamma(\frac{1}{\alpha}-\frac{1}{a})}{a},$ provided $0<\alpha<min\{1,a\}.$ 
	   		\end{itemize}
	   	\end{example}
   	
   	\begin{figure}[h]
   		\begin{center}
   			\subfigure[]{\label{c1}\includegraphics[height=1.8in]{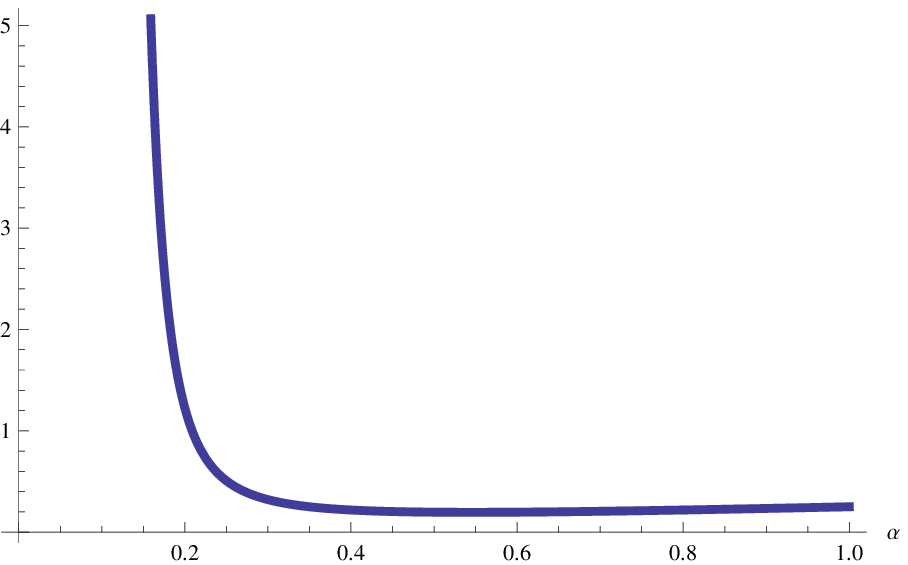}}
   			\subfigure[]{\label{c1}\includegraphics[height=1.8in]{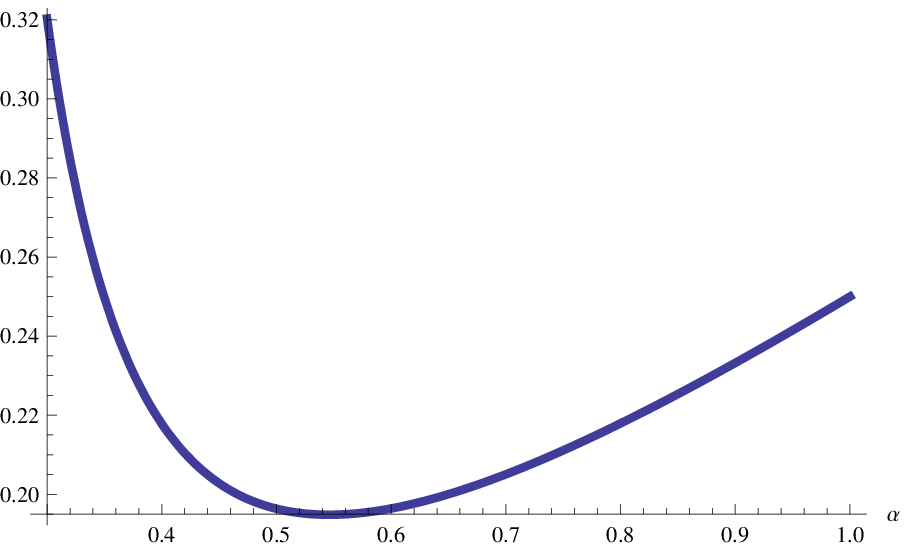}}
   			\subfigure[]{\label{c1}\includegraphics[height=1.8in]{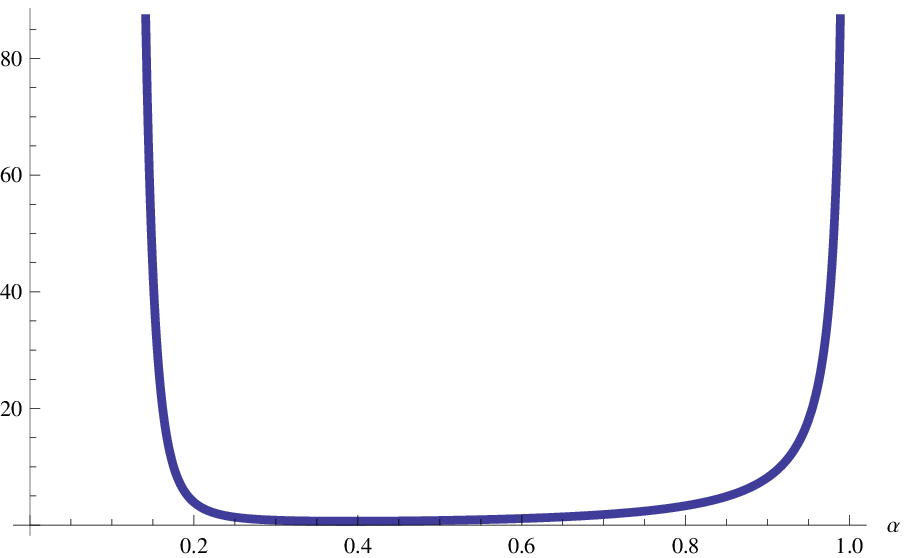}}
   			\caption{(a) Graph of EFCPE for uniform distribution in the interval $(0,1)$ as in Example \ref{ex2.1}$(i)$ when $0<\alpha<1$.  $(b)$ Magnified view of the graph as described in Figure $(a)$ when $\alpha\in (0.3,1)$. (c) Graph of the EFCPE for Fr\'echet distribution with $a=1$ and $b=1$ as in Example \ref{ex2.1}$(ii)$ when $0<\alpha<1$.}
   		\end{center}
   	\end{figure}

	   	Suppose a multi-component system is constructed in such a way that each of its components' lifetimes depend on the lifetimes of the other components. To analyze uncertainty of such system, it is required to extend the concept of uncertainty measure from univariate setup to the higher-dimensional setup.  \cite{kundu2017bivariate} proposed bivariate extension of the cumulative past entropy due to \cite{di2009cumulative} and studied its properties. Generalized version of the cumulative past entropy due to \cite{kundu2016study} was extended in the bivariate setup by \cite{kundu2018bivariate}. Along the lines of these researches, here, we propose bivariate EFCPE. Consider a random vector $(X,Y),$  where $X$ and $Y$ are non-negative random variables with respective supports $[0,s_1]$ and $[0,s_2]$. The random variables $X$ and $Y$ can be considered as the lifetimes of the components of a system having two components. Let the joint CDF of $X$  and $Y$ be $F(.,.)$. Then, the bivariate EFCPE is defined as
	   		\begin{eqnarray} \label{eq2.4}
	   	\mathcal{E}^*_\alpha(X,Y)=\int_{0}^{s_2} \int_{0}^{s_1} F(x,y)[-Ln_\alpha F(x,y)]^\frac{1}{\alpha} dxdy,~~0<\alpha< 1.
	   	\end{eqnarray}
	   	Similar to (\ref{eq2.3}), a modified bivariate EFCPE is defined as 
	   	\begin{eqnarray}
	   	\bar{\mathcal{E}}^*_\alpha(X,Y)=\int_{0}^{s_2} \int_{0}^{s_1} F(x,y)[-Ln_\alpha F(x,y)]dxdy\approx -\alpha! 	\bar{\mathcal{E}}^*(X,Y),
	   	\end{eqnarray}
	   	where $\bar{\mathcal{E}}^*(X,Y)$ is known as the bivariate cumulative past entropy (see Eq. (7) of \cite{kundu2017bivariate}).
	   	\begin{example}
	   		Let $X$ and $Y$ be the lifetimes of two components of a system with joint probability density function given by
	   		\begin{align}
	   		f(x,y)=\begin{cases}
	   		2,~~\text{if} ~0<x<1,~0<y<x\\
	   		0,~~\text{otherwise}.
	   		\end{cases}
	   		\end{align}	   
	   		Now, using the result in Example $2.1$ of \cite{kundu2017bivariate}, it can be obtained that  
	   		$$\bar{\mathcal{E}}^*_\alpha(X,Y)\approx \alpha!\left( \frac{1-\log 2}{4}\right).$$
	   	\end{example}
   	
   	In the following proposition, we present a relation between $\mathcal{E}^*_\alpha(X,Y)$ and $\bar{\mathcal{E}}^*_\alpha(X,Y).$
   	\begin{proposition}
   		Let $(X,Y)$ be a random vector, where $X$ and $Y$ are non-negative random variables with respective supports $[0,s_1]$ and $[0,s_2]$. Then, for $0<\alpha<1,$ we have $$\mathcal{E}^*_\alpha(X,Y)\ge [\bar{\mathcal{E}}^*_\alpha(X,Y)]^{\frac{1}{\alpha}}.$$
   	\end{proposition}
   \begin{proof}
   	The proof is similar to that of Proposition \ref{prop2.5}. Thus, it is omitted.
   \end{proof}
	   	
	   	Denote the conditional distribution of $Y$ given $X=x$ as $F_{Y|X=x}(y|x)=P(Y\leq y|X=x)$ and the conditional EFCPE by $\mathcal{E}^*_\alpha(Y|X)$. Below, we show that the bivariate EFCPE can be expressed in terms of the weighted EFCPE of $X$, the conditional EFCPE of $Y$ given $X$ and  the weighted conditional EFCPE of $Y$ given $X.$ In the proof, we use the following property of the fractional order logarithmic function:
	   	\begin{eqnarray}\label{eq3*}
	   		[Ln_{\alpha}uv]^{\frac{1}{\alpha}}=[Ln_{\alpha}u]^{\frac{1}{\alpha}}+[Ln_{\alpha}v]^{\frac{1}{\alpha}},~0<\alpha<1.
	   	\end{eqnarray}
	   	
	   	\begin{theorem}
	   		Suppose $X$ and $Y$ are non-negative absolutely continuous random variables with marginal CDFs $F_X(x)$ and $F_Y(y)$, respectively and $(X,Y)$ is a random vector with joint CDF $F(x,y).$ Then,
	   		\begin{eqnarray*}
	   			\mathcal{E}^*_\alpha(X,Y)=\mathcal{E}^{*,F_{Y|X=x}(y|x)}_\alpha(X)+\mathcal{E}^*_\alpha(Y|X)-\mathcal{E}^{*,\bar F_X(x)}_\alpha(Y|X),~~0<\alpha<1,
	   		\end{eqnarray*}
   		where $\mathcal{E}^{*,\bar F_X(x)}_\alpha(Y|X)=\int_{0}^{s_2}\int_{0}^{s_1}\bar{F}_X(x)F_{Y|X=x}(y|x)[-Ln_{\alpha} F_{Y|X=x}(y|x)]^{\frac{1}{\alpha}}dxdy$ is known as the weighted conditional EFCPE with weight function $\bar{F}_X(x)=1-F_X(x)$ and $\mathcal{E}^{*,F_{Y|X=x}(y|x)}_\alpha(X)=\int_{0}^{s_2}\int_{0}^{s_1}F_{Y|X=x}(y|x) F_X(x)[-Ln_\alpha F_X(x)]^\frac{1}{\alpha} dxdy$ is known as the weighted EFCPE with weight function $F_{Y|X=x}(y|x).$
	   	\end{theorem}
	   	\begin{proof}
	   		From the definition of bivariate EFCPE, we have
	   		\begin{eqnarray*}
	   			\mathcal{E}^*_\alpha(X,Y)&=&\int_{0}^{s_2} \int_{0}^{s_1} F(x,y)[-Ln_\alpha F(x,y)]^\frac{1}{\alpha} dxdy\\
	   			&=&\int_{0}^{s_2} \int_{0}^{s_1} F(x,y)[-Ln_\alpha F_X(x)]^\frac{1}{\alpha} dxdy\nonumber\\&~&+\int_{0}^{s_2} \int_{0}^{s_1} F(x,y)[-Ln_\alpha F_{Y|X=x}(y|x)]^\frac{1}{\alpha} dxdy\\
	   			&=& \int_{0}^{s_2}\int_{0}^{s_1} F_X(x)F_{Y|X=x}(y|x)[-Ln_\alpha F_X(x)]^\frac{1}{\alpha} dxdy\\
	   			&~&+\int_{0}^{s_2} \int_{0}^{s_1} F_{Y|X=x}(y|x)[-Ln_\alpha F_{Y|X=x}(y|x)]^\frac{1}{\alpha} dxdy\\
	   			&~&-\int_{0}^{s_2} \int_{0}^{s_1}\bar F_X(x) F_{Y|X=x}(y|x)[-Ln_\alpha F_{Y|X=x}(y|x)]^\frac{1}{\alpha} dxdy\\
	   			&=&\mathcal{E}^{*,F_{Y|X=x}(y|x)}_\alpha(X)+\mathcal{E}^*_\alpha(Y|X)-\mathcal{E}^{*,\bar F_X(x)}_\alpha(Y|X).
	   		\end{eqnarray*}
	   		Hence, the result follows.
	   	\end{proof}
   	Now, assume that $X$ and $Y$ are independent, that is, $F(x,y)=F_X(x)F_Y(y)$. Then, we have 
   	\begin{eqnarray}\label{eq2.5}
   		\mathcal{E}^*_\alpha(X,Y)&=&\int_{0}^{s_2} \int_{0}^{s_1} F_X(x)F_Y(y)[-Ln_\alpha F_X(x)F_Y(y)]^\frac{1}{\alpha} dxdy\nonumber\\
   		&=&  \int_{0}^{s_2} \int_{0}^{s_1} F_X(x)F_Y(y)[-Ln_{\alpha}F_X(x)]^{\frac{1}{\alpha}}dx dy\nonumber\\
   	&~&	+\int_{0}^{s_2} \int_{0}^{s_1} F_X(x)F_Y(y)[-Ln_{\alpha}F_Y(y)]^{\frac{1}{\alpha}}dx dy\nonumber\\
   		&=&\mathcal{E}^*_\alpha(X)\int_{0}^{s_2}F_Y(y)dy+\mathcal{E}^*_\alpha(Y)\int_{0}^{s_1}F_X(x)dx.
   	\end{eqnarray}
   	
   	\begin{proposition}
   	Assume that two independent random variables $X$ and $Y$ have supports $[0,s_1]$ and $[0,s_2]$, respectively. Then, we have 
   	\begin{eqnarray}\label{eq2.6}
	   	\mathcal{E}^*_\alpha(X,Y)=\mathcal{E}^*_\alpha(X)[s_2-E(Y)]+\mathcal{E}^*_\alpha(Y)[s_1-E(X)],~~0<\alpha<1.
	\end{eqnarray}
	\end{proposition}
\begin{proof}
	The proof is straightforward, and hence it is omitted.
\end{proof}

	Further, assume that $X$ and $Y$ have a common support $[0,l]$ and a common mean $\mu.$ Then, (\ref{eq2.6}) reduces to 
	\begin{eqnarray}\label{eq2.7}
	\mathcal{E}^*_\alpha(X,Y)=[l-\mu][\mathcal{E}^*_\alpha(X)+\mathcal{E}^*_\alpha(Y)],~~0<\alpha<1.
	\end{eqnarray}
	  The relation apart from the multiplicative constant $(l-\mu)$ given by (\ref{eq2.7}) is similar to the Shannon's differential entropy of two-dimensional random variable $(X,Y)$, when $X$ and $Y$ are independent.  Let $X$ and $Y$ be independent and have beta distributions with equal parameters. Then, clearly $l-\mu=\frac{1}{2}$. Thus, the bivariate EFCPE can be expressed as the arithmetic mean of the EFCPEs. 
	  \begin{proposition}
	  	Let $X_{1},\ldots,X_{n}$ be independent and identically distributed random variables with a common CDF $F(x)$ and a common mean $\mu=E(X_1)<\infty.$ Further, we assume that the random variables have a common support $[0,l]$. Then, we have 
	  	\begin{eqnarray*}
	  		\mathcal{E}^*_\alpha(X_1,\ldots,X_{n})=n[l-\mu]^{n-1}\mathcal{E}^*_\alpha(X_1),~~0<\alpha<1.
	  	\end{eqnarray*}
	  \end{proposition}
  \begin{proof}
  	The proof is simple, and thus it is omitted.  
  \end{proof}
	   	
	   	Similar to the univariate EFCPE, it can be established that the bivariate EFCPE is also a shift-independent measure. That is, for $Y_{i}=a_{i}X_{i}+b_{i}$, $i=1,~2$, $a_{i}>0$ and $b_i\ge0,$ we have 
	   	\begin{eqnarray}
	   		\mathcal{E}^*_\alpha(Y_1,Y_2)=a_1 a_2 \mathcal{E}^*_\alpha(X_1,X_2),~~0<\alpha<1.
	   	\end{eqnarray}
	   	
	   	Next, analogous to the concept of mutual information, we propose the concept of extended fractional cumulative past mutual information between two random variables $X$ and $Y.$ The mutual information between two random variables $X$ and $Y$ with joint PDF $f(x,y)$ and marginal PDFs $f_X(x)$ and $f_Y(y)$ is given by
	   	\begin{eqnarray}
	   	M(X,Y)=\int_{0}^{s_2} \int_{0}^{s_1}f(x,y)\log \frac{f(x,y)}{f_X(x)f_Y(y)}dxdy.
	   	\end{eqnarray}
	   	\cite{foroghi2022extensions} introduced the concept of fractional cumulative residual mutual information. Analogously, here we propose fractional cumulative past mutual information (FCPMI) between two random variables $X$ and $Y$ with respective marginal distribution functions $F_X(x)$ and $F_Y(y)$. 
	   	\begin{definition}\label{def2.2}
	   		Let $X$ and $Y$ with respective supports $[0,s_1]$ and $[0,s_2]$ be two non-negative absolutely continuous random variables with joint distribution function $F(x,y)$. Then, for $0<\alpha<1,$ the FCPMI between $X$ and $Y$ is given by 
	   		\begin{eqnarray}\label{eq2.10}
	   		I(X,Y)=\int_{0}^{s_2} \int_{0}^{s_1}F(x,y)\left[-Ln_{\alpha} \frac{F(x,y)}{F_{X}(x)F_{Y}(y)}\right]^{\frac{1}{\alpha}}dxdy.
	   		\end{eqnarray}
	   	\end{definition}
   	From the above definition, it is clear that the FCPMI is symmetric, nonnegative and vanishes when $X$ and $Y$ are independent. From (\ref{eq2.10}), we have 
   	\begin{eqnarray}\label{eq2.11}
   	I(X,Y)&=&\int_{0}^{s_2} \int_{0}^{s_1}F(x,y)\left[-Ln_{\alpha} F_{Y|X=x}(y|x)\right]^{\frac{1}{\alpha}}dxdy\nonumber\\&~&-\int_{0}^{s_2} \int_{0}^{s_1}F(x,y)\left[-Ln_{\alpha} F_Y(y)\right]^{\frac{1}{\alpha}}dxdy\nonumber\\
   	&=& J_{1}-J_{2}, 
   	\end{eqnarray}
   where 
   \begin{eqnarray}\label{eq2.12}
   	J_{1}=\int_{0}^{s_2} \int_{0}^{s_1}F_{Y|X=x}(y|x)F_X(x)\left[-Ln_{\alpha} F_{Y|X=x}(y|x)\right]^{\frac{1}{\alpha}}dxdy=\mathcal{E}^{*,F_X(x)}_\alpha(Y|X)
   	\end{eqnarray}
   and 
   \begin{eqnarray}\label{eq2.13}
   	J_{2}=\int_{0}^{s_2} \int_{0}^{s_1}F_Y(y)F_{X|Y=y}(x|y)\left[-Ln_{\alpha} F_Y(y)\right]^{\frac{1}{\alpha}}dxdy=s_1 \mathcal{E}^*_\alpha(Y)-\mathcal{E}^{*,E(X|Y)}_\alpha(Y).
   \end{eqnarray}
Thus, using (\ref{eq2.11})-(\ref{eq2.13}), it is easy to get the following proposition.
\begin{proposition}
 For the random variables $X$ and $Y$ as in Definition \ref{def2.2}, the FCPMI between $X$ and $Y$ is represented as
\begin{eqnarray*}
M(X,Y)=\mathcal{E}^{*,F_X(x)}_\alpha(Y|X)-s_1 \mathcal{E}^*_\alpha(Y)+\mathcal{E}^{*,E(X|Y)}_\alpha(Y).
\end{eqnarray*}
\end{proposition}
	 
Next, we propose a result which shows that the EFCPE of $X$ is expressed in terms of the reversed hazard rate $r(.)$ of $X$. We recall that $r(u)=f(u)/F(u),$ for the values of $u$ such that $F(u)$ is a strictly positive real number.

\begin{proposition}\label{prop2.4}
	Suppose $X$ is a non-negative absolutely continuous random variable with finite EFCPE. Then, for $0<\alpha<1,$
	\begin{eqnarray} 
		\mathcal{E}^*_\alpha(X)=E(\tau_{\alpha}(X))\approx(\alpha!)^{\frac{1}{\alpha}}E(\tau_{\alpha}^*(X)),
	\end{eqnarray}
where $\tau_{\alpha}(t)=\int_{t}^{\infty}(-Ln_{\alpha} F(x))^{\frac{1}{\alpha}}dx$ and $\tau_{\alpha}^{*}(t)=\int_{t}^{\infty}[-\log F(x)]^{\frac{1}{\alpha}}dx=\int_{t}^{\infty}\left[\int_{x}^{\infty}r(t)dt\right]^{\frac{1}{\alpha}}dx.$
\end{proposition}

\begin{proof}
	Using $F(x)=\int_{0}^{x}f(t)dt$, from (\ref{eq2.1}), we have
	\begin{eqnarray}\label{eq2.15}
	\mathcal{E}^*_\alpha(X)&=&\int_{0}^{\infty}\int_{0}^{x}f(t)[-Ln_{\alpha} F(x)]^{\frac{1}{\alpha}}dt dx\nonumber\\
	&=&\int_{0}^{\infty}\left(\int_{t}^{\infty}[-Ln_{\alpha} F(x)]^{\frac{1}{\alpha}}dx\right)f(t)dt\nonumber\\
	&\approx&(\alpha!)^{\frac{1}{\alpha}}\int_{0}^{\infty}\left(\int_{t}^{\infty}[-\log F(x)]^{\frac{1}{\alpha}}dx\right)f(t)dt,
	\end{eqnarray}
where the second equality in (\ref{eq2.15}) is obtained using Fubini's theorem and the final approximation is due to $Ln_{\alpha} u\approx\log u^{\alpha!}$, for $0<\alpha<1$. Thus, the result follows. 
\end{proof}
 Now, we evaluate the approximate numerical values of $\mathcal{E}^*_\alpha(X)$ and $\mathcal{\bar E}^*_\alpha(X)$, respectively given by (\ref{eq2.1}) and (\ref{eq2.3}) for some specific values of $\alpha$ of uniform distribution in support $(0,1)$. The numerical values are given in Table $1,$ which show that $\mathcal{E}^*_\alpha(X)$ and $\mathcal{\bar E}^*_\alpha(X)$ do not have any inequality in general. In the following result, we establish an inequality between $\mathcal{E}^*_\alpha(X)$ and $[\bar{\mathcal{E}}^*_\alpha(X)]^{\frac{1}{\alpha}}$.
	\begin{table}[h!]
		\centering
		\begin{tabular}{ | m{1cm} | m{1.2cm}| m{1.2cm} | m{1.2cm} | m{1.2cm} | m{1.2cm} | m{1.2cm} | m{1.2cm} | m{1.2cm}| m{1.2cm} | }
			\hline\vspace{.1cm}
		$\alpha$ & 0.1 & 0.2 & 0.3 & 0.4& 0.5 & 0.6 & 0.7 & 0.8 & 0.9 \\
			\hline\vspace{.1cm}
			$\mathcal{E}^*_\alpha(X)$ & 1076.07 &  1.22353 & 0.32030 & 0.21782 &0.08701 & 0.19640 & 0.20388 & 0.21793 & 0.23322 \\ 
			\hline\vspace{.1cm}
			$\mathcal{\bar{E}}^*_\alpha(X)$ & 0.23784 & 0.22954 & 0.22436 & 0.22182 & 0.22156 & 0.22338 & 0.22716 & 0.23285 & 0.24044 \\ 
			\hline
		\end{tabular}
		\caption{Approximate values of $\mathcal{E}^*_\alpha(X)$ (see Example \ref{ex2.1}) and $\mathcal{\bar E}^*_\alpha(X)(\approx \frac{\alpha!}{4})$ for uniform distribution in the interval $(0,1)$, for some specific values of $\alpha$.}
	\end{table}

\begin{proposition}\label{prop2.5}
	Let $X$ be a non-negative absolutely continuous random variable with  $\mathcal{E}^*_\alpha(X)<\infty.$ Then, for $0<\alpha<1,$ we have
$\mathcal{E}^*_\alpha(X)\ge [\bar{\mathcal{E}}^*_\alpha(X)]^{\frac{1}{\alpha}},$ where $\bar{\mathcal{E}}^*_\alpha(X)$ is given by (\ref{eq2.3}).
\end{proposition}
\begin{proof}
To prove the proposition, we note that for $0<\alpha<1$, $F(x)\ge [F(x)]^{\frac{1}{\alpha}}$ holds. Thus,
		\begin{eqnarray}\label{eq2.16}
		\mathcal{E}^*_\alpha(X)=\int_{0}^{\infty}  F(x)[-Ln_\alpha F(x)]^\frac{1}{\alpha} dx\geq\int_{0}^{\infty}  [-F(x)Ln_\alpha F(x)]^\frac{1}{\alpha} dx.
		\end{eqnarray}
		Moreover, for $0<\alpha<1,$ it is easy to show that $\phi(x)=x^{\frac{1}{\alpha}}$ is convex with respect to $x$. Using this and the Jensen's inequality in (\ref{eq2.16}), the desired result easily follows. 
\end{proof}

Using similar arguments as in Proposition \ref{prop2.5}, one can obtain that 
\begin{eqnarray}
\bar{\mathcal{E}}^*_\alpha(X)=E[W_{\alpha}(X)], 
\end{eqnarray}
where $W_{\alpha}(t)=-\int_{t}^{\infty}Ln_{\alpha} F(x)dx.$ Differentiating $W_{\alpha}(t)$ with respect to $t$ twice, we obtain 
\begin{eqnarray}\label{eq2.18}
W_{\alpha}^{\prime\prime}(t) =\frac{d}{dt}Ln_{\alpha} F(t).
\end{eqnarray}
It is known that $Ln_{\alpha}x$ is the inverse of MLF, say $g(x)$, that is, $Ln_{\alpha}x=g^{-1}(x)$, implies  $Ln_{\alpha}F(x)=g^{-1}(F(x))$. Thus, from (\ref{eq2.18}), for $\alpha\in(0,1)$, we obtain 
\begin{eqnarray}
	W_{\alpha}^{\prime\prime}(t)&=&\frac{d}{dt} g^{-1}(F(t))\nonumber\\
	&=& \frac{f(t)}{g^{\prime}[g^{-1}(F(t))]},
\end{eqnarray}
which is clearly non-negative, since $g^{\prime}(.)\ge0.$ Thus, $W_{\alpha}(t)$ is convex with respect to $t>0.$ This observation yields a lower bound of  $\bar{\mathcal{E}}^*_\alpha(X)$, which is given by 
\begin{eqnarray}
\bar{\mathcal{E}}^*_\alpha(X)\ge W_{\alpha}(\mu),
\end{eqnarray}
where $\mu=E(X).$

Gini index is well-known in social welfare studies for income inequality. It is also well-known that for a more polarized society, the value of Gini index must be higher. The Gini index of a distribution with mean $\mu=E(X)<\infty$ is defined as 
\begin{eqnarray}
Gini(X)=1-\frac{\int_{0}^{\infty}\bar{F}^{2}(x)dx}{\int_{0}^{\infty}\bar{F}(x)dx}=1-\frac{\int_{0}^{\infty}\bar{F}^{2}(x)dx}{\mu}.
\end{eqnarray} 
Below, we obtain a lower bound of the modified EFCPE given by (\ref{eq2.2}). 
\begin{proposition}
	For a non-negative absolutely continuous random variable $X$ with finite mean $\mu$, we have 
	\begin{eqnarray}
	\bar{\mathcal{E}}^*_\alpha(X)\ge \alpha!\mu Gini(X), ~~0<\alpha<1.
	\end{eqnarray}
\end{proposition}
\begin{proof}
	From (\ref{eq2.3}), we have 
	\begin{eqnarray}
		\mathcal{\bar E}^*_\alpha(X)&=&\int_{0}^{\infty}  F(x) [-\text{Ln}_\alpha F(x)]dx\nonumber\\
		&\approx& -\alpha!\int_{0}^{\infty} F(x) \log F(x) dx\nonumber\\
		&=&	\alpha!\int_{0}^{\infty} F(x) |\log F(x)| dx.
	\end{eqnarray}
	Now, the rest of the proof follows using the inequality $x(1-x)\le x|\log x|,$ for $0<x<1.$ Thus, it is omitted. 
\end{proof}

Various stochastic orders have been proposed in the literature in order to compare  distributions. In the following, we find some relationships between existing stochastic orderings and the uncertainty ordering on the basis of the newly proposed measure given by (\ref{eq2.1}). A non-negative random variable $X$ with CDF $F(.)$ and PDF $f(.)$ is said to be smaller than $Y$ with CDF $G(.)$ and PDF $g(.)$ in the sense of 
\begin{itemize}
\item dispersive ordering,  denoted by $X\leq^{disp}Y$ if $f(F^{-1}(v))\geq g(G^{-1}(v)),$ for all $v\in (0,1)$, where $F^{-1}(.)$ and $G^{-1}(.)$ are right continuous inverses of $F(.)$ and $G(.)$, respectively;

\item decreasing convex order, denoted by $X\leq^{dcx}Y$ if $E(\phi(X))\leq E(\phi(Y)),$ for all decreasing convex functions $\phi(.)$.
\end{itemize}
For details, please refer to \cite{shaked2007stochastic}.

\begin{theorem}\label{th2.2}
	For two non-negative absolutely continuous random variables $X$ and $Y,$ we have
	\begin{itemize}
		\item[(i)] $X\leq^{disp} Y\Rightarrow \mathcal{E}^*_\alpha(X)\le \mathcal{E}^*_\alpha(Y);$
		\item[(ii)] $X\leq^{dcx} Y\Rightarrow \bar{\mathcal{E}}^*_\alpha(X)\le \bar{\mathcal{E}}^*_\alpha(Y),$
	\end{itemize}
where $0<\alpha<1.$
\end{theorem}
\begin{proof}~~
		$(i)$ The proof is analogous to Lemma $3$ of \cite{klein2016cumulative}, and thus it is omitted. \\
		$(ii)$ The proof follows using the fact that the function $W_{\alpha}(x)$ is decreasing convex with respect to $x$. 
\end{proof}

The result in Theorem \ref{th2.2}$(i)$ ensures that the newly proposed EFCPE can be considered as a dispervive measure. The following example is an illustration of Theorem \ref{th2.2}$(i)$.

\begin{example}
	Consider two random variables $X$ and $Y$ with respective CDFs $F(x)=1-(1+x)^{-k_{1}},~x>0,~k_{1}>0$ and $G(x)=1-(1+x)^{-k_{2}},~x>0,~k_{2}>0$, with $k_1>k_2.$ Then, it is not hard to see that $X$ is smaller than $Y$ in hazard rate ordering. For details on hazard rate ordering, please refer to \cite{shaked2007stochastic}. Further, $X$ has decreasing failure rate. Thus, due to \cite{bagai1986tail}, it can be concluded that $X\le^{disp} Y.$ Now, 
		\begin{eqnarray}
	\mathcal{E}^*_\alpha(X)\approx(\alpha!)^{\frac{1}{\alpha}}\int_{0}^{\infty}\left(1-\frac{1}{(1+x)^{k_1}}\right)\left[-\log \left(1-\frac{1}{(1+x)^{k_1}}\right)\right]^{\frac{1}{\alpha}}dx
	\end{eqnarray}
	and
		\begin{eqnarray}
	\mathcal{E}^*_\alpha(Y)\approx(\alpha!)^{\frac{1}{\alpha}}\int_{0}^{\infty}\left(1-\frac{1}{(1+x)^{k_2}}\right)\left[-\log \left(1-\frac{1}{(1+x)^{k_2}}\right)\right]^{\frac{1}{\alpha}}dx.
	\end{eqnarray}
	In order to validate the result in Theorem \ref{th2.2}$(i)$, we present some values of $\mathcal{E}^*_\alpha(X)$ and $\mathcal{E}^*_\alpha(Y)$, for some values of $\alpha$ in Table $2.$ Here, we have assumed that $k_1=0.7$ and $k_2=0.5.$
	\begin{table}[h!]
		\centering
		\begin{tabular}{ | m{2.5em} | m{.8cm}| m{.8cm} | m{.8cm} | m{1.5cm} | m{2.1cm} | m{2.1cm} | m{2.1cm} | } 
			\hline\vspace{.1cm}
			$\alpha$ & 0.2 & 0.3 & 0.4 &0.5&0.6&0.8&1\\ 
			\hline\vspace{.1cm}
			$\mathcal{E}^*_\alpha(Y)$ & 3.54 &  1.72 &3.13 & $1.6\times10^{4}$ &$1.02\times 10^{1169}$&$4.2\times 10^{2624}$&$1.02\times 10^{3498}$\\ 
			\hline\vspace{.1cm}
			$\mathcal{E}^*_\alpha(X)$ & 2.31 & 0.93 & 1.06  &1.82&4.63&$1.03\times 10^{441}$&$4.5\times 10^{2100}$\\ 
			\hline
		\end{tabular}
		\caption{Approximate values for $\mathcal{E}^*_\alpha(X)$ and $\mathcal{E}^*_\alpha(Y)$, for some specific values of $\alpha.$}
	\end{table}
\end{example}


We recall that the Shannon entropy of the sum of two independent random variables is larger than that of either. Similar observation was noticed by \cite{rao2004cumulative} and \cite{di2009cumulative} for cumulative residual entropy and cumulative past entropy, respectively. Here, in the next theorem, we establish a similar result for EFCPE. The proof is similar to that of Theorem $3.2$ of \cite{di2017further}. Thus, we omit it. 

 \begin{proposition}\label{pro2.12}
 	For non-negative and independent random variables $X$ and $Y$ with respective CDFs $F(.)$ and $G(.),$ we have
 	\begin{eqnarray}
 		\mathcal{E}^*_\alpha(X+Y)\geq max\{	\mathcal{E}^*_\alpha(X),	\mathcal{E}^*_\alpha(Y)\},
 	\end{eqnarray}
 if $X$ and $Y$ have log-concave density functions. 
 \end{proposition}

\begin{example}\label{ex2.4}
	Let $X$ and $Y$ be two independent random variables with a common CDF $F(x)=x,~0<x<1$. Then, the CDF of $Z=X+Y$ can be obtained as
	\begin{align}
	K(x)=\begin{cases}
	\frac{x^2}{2},~~\text{if} ~0<x\le1,\\
	1-\frac{(x-2)^{2}}{2},~~\text{if}~1\le x<2.
	\end{cases}
	\end{align}	   
	Thus, from (\ref{eq2.2})
	\begin{eqnarray}
	\mathcal{E}^*_\alpha(X+Y)&\approx& (\alpha!)^{\frac{1}{\alpha}}\int_{0}^{1}\frac{x^2}{2}\left[-\log \frac{x^2}{2}\right]^{\frac{1}{\alpha}}dx\nonumber\\
	&~&+(\alpha!)^{\frac{1}{\alpha}}\int_{1}^{2}\left(1-\frac{(x-2)^2}{2}\right)\left[-\log \left(1-\frac{(x-2)^2}{2}\right)\right]^{\frac{1}{\alpha}}dx
	\end{eqnarray}
	and 
	\begin{eqnarray}
		\mathcal{E}^*_\alpha(X)~or~\mathcal{E}^*_\alpha(Y)&\approx&(\alpha!)^{\frac{1}{\alpha}}\int_{0}^{1}x\left[-\log x\right]^{\frac{1}{\alpha}}dx=\frac{(\alpha!)^{1/\alpha}\Gamma(\frac{1}{\alpha}+1)}{2^{\frac{1}{\alpha}+1}}.
	\end{eqnarray}
	Further, in order to validate Proposition \ref{pro2.12}, we present the values of $\mathcal{E}^*_\alpha(X+Y)$ and max\{$\mathcal{E}^*_\alpha(X),\mathcal{E}^*_\alpha(Y)$\} for some specific values of $\alpha$ in Table $3.$
\end{example}
\begin{table}[h!]
	\centering
	\begin{tabular}{ | m{9em} | m{1cm}| m{1cm} | m{1cm} | m{1cm} | m{1cm} | m{1cm} | m{1cm} | m{1cm} | } 
		\hline\vspace{.1cm}
		$\alpha$ & 0.2 & 0.3 & 0.4 &0.5&0.6&0.8&0.9&1\\ 
		\hline\vspace{.1cm}
		$\mathcal{E}^*_\alpha(X+Y)$ & 4.86 & 0.79 &0.43 & 0.34&0.31&0.32&0.34&0.36\\ 
		\hline\vspace{.1cm}
		max\{$\mathcal{E}^*_\alpha(X),\mathcal{E}^*_\alpha(Y)$\} & 1.22 & 0.32 & 0.22  &0.20&0.19&0.22&0.23&0.24\\ 
		\hline
	\end{tabular}
\caption{Approximate values for $\mathcal{E}^*_\alpha(X+Y)$ and max\{$\mathcal{E}^*_\alpha(X),\mathcal{E}^*_\alpha(Y)$\}, for some specific values of $\alpha$ as in Example \ref{ex2.4}.}
\end{table}

Next, we consider proportional reversed hazard rate model and obtain the EFCPE. Let $X_{\delta}$ and $X$ be two non-negative random variables with respective CDFs $F_{\delta}(.)$ and $F(.)$. Further, assume that they have proportional reversed hazard rate model, that is, $F_{\delta}(x)=[F(x)]^{\delta},$ for some constant $\delta>0.$ Using the relation $Ln_{\alpha}(u^{c})=c^{\alpha}Ln_{\alpha}(u)$, for $0<\alpha<1,$ the EFCPE of $X_{\delta}$ can be written as 
\begin{eqnarray}
\mathcal{E}^*_\alpha(X_{\delta})=\delta\int_{0}^{\infty}[F(x)]^{\delta}[-Ln_{\alpha} F(x)]^{\frac{1}{\alpha}},~~0<\alpha<1.
\end{eqnarray}
Now, let $\delta\ge1.$ Then, $[F(x)]^{\delta}\le F(x),$ which implies that 
\begin{eqnarray}\label{eq2.21}
\mathcal{E}^*_\alpha(X_{\delta})\le \delta \mathcal{E}^*_\alpha(X).
\end{eqnarray}
Let $\delta=n$ be a natural number, where $n>1.$ Further, let $X_1,\ldots,X_n$ be the component lifetimes of a parallel system, independently distributed with a common CDF $F(.)$. Then, $F_{\delta}(x)$ represents the CDF of the lifetime of a parallel system. So, it is easy to observe that (\ref{eq2.21}) is useful to get an upper bound of the EFCPE of the lifetime of a parallel system. 

\subsection{Conditional EFCPE}
This subsection focuses on the development of the conditional EFCPE and its properties. Let $(\Omega,\mathcal{F},\mathcal{P})$ be a probability space and a non-negative absolutely continuous random variable $X$ is defined on it. Here, $\Omega$ is the sample space, $\mathcal{F}$ is the $\sigma$-field of subsets of $\Omega$ and $\mathcal{P}$ is the probability measure. Further, we denote the conditional expectation of $X$ given a sub $\sigma$-field $\mathcal{G}$ as $E(X|\mathcal{G})$, where $\mathcal{G}\subset \mathcal{F}.$ In this following definition, we present the conditional EFCPE of $X$.
\begin{definition}
	Suppose a non-negative absolutely continuous random variable $X$ has CDF $F(.)$. Then, the conditional EFCPE for given a $\sigma$-field $\mathcal{F}$ is defined as
	\begin{eqnarray*}
		\mathcal{E}^*_\alpha(X|\mathcal{F})&=&\int_{\mathcal{R}^+ }P(X\leq x|\mathcal{F})[-Ln_\alpha( P(X\leq x|\mathcal{F}))]^\frac{1}{\alpha} dx\\
		&\approx& (\alpha!)^\frac{1}{\alpha}\int_{\mathcal{R}^+ }E[I_{X\leq x}|\mathcal{F}][-\log( E[I_{X\leq x}|\mathcal{F}])]^\frac{1}{\alpha} dx,~~0<\alpha<1,
	\end{eqnarray*}
	where $I_{X\leq x}$ is an indicator function.
\end{definition}
We remark that $\mathcal{E}^*_\alpha(X|\mathcal{F})$ measures the uncertainty of a random variable $X$ with respect to $\mathcal{F}.$ For instance, assume that a $\sigma$-field $\mathcal{F}$ has been generated by another random variable $Y$. Then, we have 
\begin{eqnarray}
\mathcal{E}^*_\alpha(X|\mathcal{F})=K_1(Y)=\int_{\mathcal{R}^+ }P(X\leq x|Y=y)[-Ln_\alpha( P(X\leq x|Y=y))]^\frac{1}{\alpha} dx.
\end{eqnarray} 

The conditional version of the modified EFCPE given by (\ref{eq2.3}) is defined as 
	\begin{eqnarray}\label{eq2.35*}
	\bar{\mathcal{E}}^*_\alpha(X|\mathcal{F})&=&\int_{\mathcal{R}^+ }P(X\leq x|\mathcal{F})[-Ln_\alpha( P(X\leq x|\mathcal{F}))]dx,~~0<\alpha<1,
\end{eqnarray}
Now, suppose $\mathcal{F}$ is a trivial field, that is, $\mathcal{F}=\{\emptyset,\Omega\}$. Then, it can be shown that 
\begin{itemize}
	\item $\mathcal{E}^*_\alpha(X|\mathcal{F})=\mathcal{E}^*_\alpha(X)$;
	\item $\bar{\mathcal{E}}^*_\alpha(X|\mathcal{F})=\bar{\mathcal{E}}^*_\alpha(X)$.
\end{itemize}

The following result provides a bound of the conditional EFCPE in terms of the measure defined in  (\ref{eq2.35*}).
 \begin{proposition}\label{prop2.8}
	For a non-negative and absolutely continuous random variable $X$ with CDF $F(.)$, we have  
	\begin{eqnarray}\label{eq2.24}
		\mathcal{E}^*_\alpha(X|\mathcal{F})\geq [\bar{\mathcal{E}}^*_\alpha(X|\mathcal{F})]^{\frac{1}{\alpha}},  ~0<\alpha<1.
	\end{eqnarray}
\end{proposition}
\begin{proof}
	The proof is analogous to that of Proposition \ref{prop2.5}, and thus it is not presented here. 
\end{proof}

\begin{proposition}
	Consider a Markov chain $U\rightarrow V\rightarrow W.$ Then, 

		$$\mathcal{E}^*_\alpha(W|V,U)=\mathcal{E}^*_\alpha(W|V).$$
	
\end{proposition}
\begin{proof}
	The proof follows using the concept of Markovian property. Thus, it is omitted.
\end{proof}

The next result explores the condition, under which the expected conditional EFCPE vanishes. For the concept of $\mathcal{F}$-measurable, please refer to \cite{rao2004cumulative}.
\begin{theorem}
	For finite $E(|X|^p)$,  $p\geq1$ and a $\sigma$-field $\mathcal{F}$, we have $ E(\mathcal{E}^*_\alpha(X|\mathcal{F}))=0$ if and only if X is $\mathcal{F}$-measurable.
\end{theorem}
\begin{proof}
	We omit the proof since it is similar to Theorem 3.5 of \cite{foroghi2022extensions}. 
\end{proof}

\begin{theorem}
	Let $X$ be any random variable and $\mathcal{F}$ be a  $\sigma$-field. Then, we get $ E(\mathcal{E}^*_\alpha(X|\mathcal{F}))\leq \mathcal{E}^*_\alpha(X|\mathcal{F}),  for ~0<\alpha<1$ and equality holds iff $X$ is independent of $\mathcal{F}$.
\end{theorem}
\begin{proof}
	The proof is analogous to Theorem $7$ of \cite{rao2004cumulative}, and thus it is omitted. 
\end{proof}

\subsection{Dynamic version of  EFCPE}
For modelling lifetime data, the concepts of residual and past lifetimes have been widely used by several researchers. In reliability theory, the residual lifetime means the additional lifetime of a system given that the system has survived until time $t$. The past lifetime is a dual concept of the residual lifetime. Suppose the system has already failed at time $t>0$. Then, the past lifetime, denoted by $X_{t}=[t-X|X\le t]$, represents the time elapsed after failure till time $t$. \cite{di2021fractional} introduced dynamic fractional generalized cumulative residual entropy for the residual lifetime (see Eq. (28)).  \cite{foroghi2022extensions} proposed extended fractional cumulative residual entropy for residual lifetime. In this flow of research, here we consider EFCPE for past lifetime and study some properties. 
\begin{definition}\label{def2.4}
Let $X$ be the lifetime of a system with CDF $F(.)$ and $X_t=[t-X|X\leq t]$ be the past lifetime with CDF $F_{X_{t}}(x)=\frac{F(x)}{F(t)},~x<t$. Then, for $0<\alpha<1,$ the dynamic EFCPE is 
\begin{equation}\label{eq2.25}
	\mathcal{E}^*_\alpha(X_t)=\mathcal{E}^*_\alpha(X;t)=\int_{0}^{t}  \frac{F(x)}{F(t)}\bigg[-Ln_\alpha \frac{F(x)}{F(t)}\bigg]^\frac{1}{\alpha} dx
	,~t>0,
\end{equation}

\end{definition}
We note that (\ref{eq2.25}) can be approximately written as  
\begin{equation}\label{eq2.25*}
\mathcal{E}^*_\alpha(X_t)\approx(\alpha!)^{\frac{1}{\alpha}}\int_{0}^{t}  \frac{F(x)}{F(t)}\bigg[-\log \frac{F(x)}{F(t)}\bigg]^\frac{1}{\alpha} dx,~t>0,
\end{equation}

Note that when $t$ tends to infinity, then the dynamic EFCPE reduces to the EFCPE given in Definition \ref{def1.1}. Similar to the concept of dynamic EFCPE, the dynamic version of modified EFCPE is given by 
\begin{equation}\label{eq2.26}
	\mathcal{\bar E}^*_\alpha(X_t)=\mathcal{\bar E}^*_\alpha(X;t)=-\int_{0}^{t}  \frac{F(x)}{F(t)}\bigg[Ln_\alpha \frac{F(x)}{F(t)}\bigg] dx\approx -\alpha! \int_{0}^{t}  \frac{F(x)}{F(t)}\bigg[\log \frac{F(x)}{F(t)}\bigg] dx,~0<\alpha<1.
\end{equation}
We have the following observations, which are similar to EFCPE. 
\begin{itemize}
	\item Suppose that $X$ is a random variable with support $[0,a]$, and symmetric with respect to $\frac{a}{2},$ that is, $F(x)=\bar F(a-x),$ for all $0\leq x \leq a.$ Then, 
	\begin{eqnarray*}
		\mathcal{E}^*_\alpha(X;t)=\mathcal{E}_\alpha(X;a-t),~0\leq t\leq a,~0<\alpha<1,
	\end{eqnarray*}
	where $\mathcal{E}_\alpha(X;a-t)=\int_{a-t}^{a}  \frac{\bar F(x)}{\bar F(a-t)}\big[-Ln_\alpha \frac{\bar F(x)}{\bar F(a-t)}\big]^\frac{1}{\alpha} dx$ is known as the dynamic fractional cumulative residual entropy.
	\item Consider $Y=aX+b$ with $a>0$ and $b\geq 0$. Then, we obtain
	\begin{eqnarray*}
		\mathcal{E}^*_\alpha(Y;t)=a\mathcal{E}^*_\alpha\bigg(X;\frac{t-b}{a}\bigg),~t\geq b,~0<\alpha<1.
	\end{eqnarray*}
\end{itemize}

The following proposition presents an alternative way of representation of the dynamic EFCPE.

\begin{proposition}
	Let $X$ be a non-negative absolutely continuous random variable with CDF $F(.)$. Then, 
	\begin{eqnarray*}
		\mathcal{E}^*_\alpha(X;t)=E[T_{\alpha}(X;t)|X\leq t]\approx (\alpha!)^{\frac{1}{\alpha}}E[T^*_{\alpha}(X;t)|X\leq t],~0<\alpha<1,
	\end{eqnarray*}
	where $T_{\alpha}(x;t)=\int_{x}^{t} [-Ln_\alpha \frac{F(z)}{F(t)}]^\frac{1}{\alpha} dz$ and $T^*_{\alpha}(x;t)=\int_{x}^{t} [-\log \frac{F(z)}{F(t)}]^\frac{1}{\alpha} dz$.
\end{proposition}
\begin{proof}
	The proof is similar to Proposition \ref{prop2.4}. Thus, it is not presented here.
\end{proof}

Next, we obtain bounds of the dynamic EFCPE. 
\begin{proposition}\label{prop2.11}
	Let $X$ be a non-negative absolutely continuous random variable with CDF $F(.)$. Then, for $0<\alpha<1,$ we have
	\begin{itemize}
		\item[(i)] $\mathcal{E}^*_\alpha(X;t)\ge [\bar{\mathcal{E}}^*_\alpha(X;t)]^{\frac{1}{\alpha}},$ where $\bar{\mathcal{E}}^*_\alpha(X;t)$ is given by (\ref{eq2.26});
		\item[(ii)] $
			\mathcal{E}^*_\alpha(X;t)\geq \mu(t)[-Ln_\alpha(1/F(t))]^\frac{1}{\alpha}.$
			\end{itemize}
			
\end{proposition}
\begin{proof}~~
\begin{itemize}
	\item[(i)] The proof of Part (i) follows easily from that of Proposition \ref{prop2.5}.
	\item[(ii)] Making use of (\ref{eq3*}) , the dynamic EFCPE given by (\ref{eq2.25}) can be rewritten as 
	\begin{eqnarray}\label{eq2.39*}
		\mathcal{E}^*_\alpha(X;t)=\frac{1}{F(t)}\int_{0}^{t}  F(x)[-Ln_\alpha F(x)]^{\frac{1}{\alpha}} dx + [-\text{Ln}_\alpha (1/F(t))]^{\frac{1}{\alpha}}\mu(t),
	\end{eqnarray}
	where $\mu(t)=\int_{0}^{t}\frac{F(x)}{F(t)}dx$ is known as the mean inactiving time of $X$. Moreover, the first integral term in the right hand side of (\ref{eq2.39*}) is non-negative. Thus, we have 
	$$\mathcal{E}^*_\alpha(X;t)\geq \mu(t)[-Ln_\alpha(1/F(t))]^\frac{1}{\alpha}.$$	
\end{itemize}
Hence, the theorem is proved.
\end{proof}

Let $X_{\delta}$ and $X$ be two non-negative random variables with respective CDFs $F_{\delta}(.)$ and $F(.)$, satisfying proportional reversed hazard rate model. Then, we obtain
\begin{eqnarray*}
	\mathcal{E}^*_\alpha(X_\delta;t)\leq \delta   \mathcal{E}^*_\alpha(X;t),~\mbox{for} ~\delta\geq 1.
\end{eqnarray*}

\section{Extended fractional cumulative paired $\phi$-entropy} \setcounter{equation}{0}
We note that the $\phi$-entropy was introduced by \cite{burbea1982convexity}. Recently, \cite{klein2016cumulative} mentioned with some reasons that it is doubtful for  $\phi$-entropy to be a dispersive measure. These authors have introduced a generalized measure, known as the cumulative $\phi$-entropy. In this section, we propose extended fractional cumulative paired $\phi$-entropy.

\begin{definition}
	Let $X$ be a non-negative absolutely continuous random variable with CDF $F(.)$ and reliability function $\bar{F}(.)$. Then, the extended fractional cumulative paired $\phi$-entropy is defined as 
	\begin{eqnarray*}
		\mathcal{PE}_\alpha(X)&=&\int_{0}^{\infty}  \bar F(x)[-Ln_\alpha \bar F(x)]^\frac{1}{\alpha}dx+\int_{0}^{\infty}  F(x)[-Ln_\alpha F(x)]^\frac{1}{\alpha} dx\\
		&=&\mathcal{E}_\alpha(X)+\mathcal{E}^*_\alpha(X), ~~0<\alpha<1,
	\end{eqnarray*}
	where $\mathcal{E}_\alpha(X)=\int_{0}^{\infty}  \bar F(x)[-Ln_\alpha \bar F(x)]^\frac{1}{\alpha}dx$ is the extended fractional cumulative residual entropy.
\end{definition}

	 By using the approximation $Ln_\alpha p\approx \log p^{\alpha!}$, $0<\alpha<1,$ we can obtain an extended version of a modified extended fractional cumulative paired $\phi$-entropy as 
\begin{eqnarray*}
	\mathcal{PE}^*_\alpha(X)&\approx&(\alpha!)^\frac{1}{\alpha}\bigg[\int_{0}^{\infty}  \bar F(x)[-\log \bar F(x)]^\frac{1}{\alpha}dx+\int_{0}^{\infty}  F(x)[-\log F(x)]^\frac{1}{\alpha} dx\bigg],~0<\alpha<1.
\end{eqnarray*}

Next, we obtain extended fractional cumulative paired $\phi$-entropy for the affine transformation $Y=aX+b,~a>0$ and $b\ge0.$
\begin{proposition}
	Suppose $X$ is a non-negative absolutely continuous random variable with CDF $F(.)$ and survival function $\bar F(.)$. Then, for $a>0$ and $b\ge0,$
	\begin{eqnarray} 
		\mathcal{PE}_\alpha(aX+b)=|a|\mathcal{PE}_\alpha(X),~~0<\alpha<1.
	\end{eqnarray}
\end{proposition}
\begin{proof}
	It is already observed that $\mathcal{E}^*_\alpha(aX+b)=|a|\mathcal{E}^*_\alpha(X).$ Utilizing this and Proposition 2.7 of \cite{foroghi2022extensions}, we obtain
	\begin{eqnarray*}
		\mathcal{PE}_\alpha(aX+b)&=& \mathcal{E}_\alpha(aX+b)+ \mathcal{E}^*_\alpha(aX+b)\\
		&=&|a| \mathcal{E}_\alpha(X)+|a|\mathcal{E}^*_\alpha(X)\\
		&=&|a|( \mathcal{E}_\alpha(X)+\mathcal{E}^*_\alpha(X)),
	\end{eqnarray*} 
	which completes the proof. 
\end{proof}

The following proposition shows that the dispersive order between two distributions preserves the uncertainty order on the basis of the cumulative paired $\phi$-entropy. 
\begin{proposition}
	Let $X$ and $Y$ be two non-negative absolutely continuous random variables with CDFs $F(.)$ and $G(.)$ and its survival functions $\bar F(.)$ and $\bar G(.),$ respectively. Then, $X\leq^{disp}Y\Rightarrow \mathcal{PE}_\alpha(X)\leq  \mathcal{PE}_\alpha(Y), ~0<\alpha<1$.
\end{proposition}
\begin{proof}
	Using Theorem \ref{th2.2}$(i)$ and Proposition $2.16$ of \cite{foroghi2022extensions}, the desired result follows. Thus, the details are omitted. 
\end{proof}

Next, we establish a result similar to Proposition \ref{pro2.12}. The proof is omitted since it is analogous to Theorem $3.2$ of \cite{di2017further}.
	\begin{proposition}
	Consider $X$ and $Y$ two non-negative independent continuous random variables with CDFs $F(.)$ and $G(.)$ and survival functions $\bar F(.)$ and $\bar G(.),$ respectively. Then,
	\begin{eqnarray*}
		\mathcal{PE}_\alpha(X+Y)\geq max\{	\mathcal{PE}_\alpha(X),	\mathcal{PE}_\alpha(Y)\},~0<\alpha<1.
	\end{eqnarray*} 
\end{proposition}

\section{Empirical EFCPE}
In this section, we propose empirical estimators of the EFCPE. We consider a random sample $(X_1,\ldots,X_n)$ drawn from a population with CDF $F(.)$. The order statistics corresponding to $(X_1,\ldots,X_n)$ are denoted by
$X_{1:n}\leq X_{2:n}\leq...\leq X_{n:n},$  where 
$X_{1:n}= min(X_1,\ldots,X_n)$ and $X_{n:n}= max(X_1,\ldots,X_n).$ The empirical CDF is given by	
\begin{align}\label{eq5.1}
\tilde{F_n}(x)=\begin{cases}
0,~~if ~x<X_{1:n},\\
\frac{\mathcal{K}}{n},~~if~X_{\mathcal{K}:n}\leq x<X_{\mathcal{K}+1:n}\\
1,~~if~x\geq X_{n:n},
\end{cases}
\end{align}	   
where $\mathcal{K}=1,\ldots,n-1$. Now, making use of  (\ref{eq5.1}), the EFCPE can be written as
\begin{align}\label{eq5.2}
\mathcal{E}^*_\alpha(\tilde{F}_n)&= \int_{0}^{\infty}  \tilde{F}_n(x)[-\text{Ln}_{\alpha} \tilde{F}_n(x)]^\frac{1}{\alpha} dx
\\
&\approx (\alpha!)^\frac{1}{\alpha} \int_{0}^{\infty}  \tilde{F}_n(x)[-\log \tilde{F}_n(x)]^\frac{1}{\alpha} dx\notag\\
&=(\alpha!)^\frac{1}{\alpha}\sum_{i=1}^{n-1} \int_{X_{i:n}}^{X_{i+1:n}}  \tilde{F}_n(x)[-\log \tilde{F}_n(x)]^\frac{1}{\alpha} dx\notag\\
&=(\alpha!)^\frac{1}{\alpha}\sum_{i=1}^{n-1}U_i\left(\frac{i}{n}\right)\left(-\log \frac{i}{n}\right)^\frac{1}{\alpha},
\end{align}
where $U_i=X_{i+1:n}-X_{i:n},~ i=1,\ldots,n-1$.  The following theorem establishes that the empirical EFCPE converges to the EFCPE as $n$ tends to infinity.  
	\begin{theorem}
	Let $X$ be a non-negative random variable with CDF $F(.)$ and $E(|X|^p)<\infty$, for $p>1$. Then, we obtain
	\begin{eqnarray*}
		\mathcal{E}^*_\alpha(\tilde{F}_n)\xrightarrow{a.s} \mathcal{E}^*_\alpha(F), ~as ~n\longrightarrow\infty,
	\end{eqnarray*}
where $ \mathcal{E}^*_\alpha(F)$ is given by (\ref{eq2.1}).
\end{theorem}
\begin{proof} In order to establish the desired result, we consider 
	\begin{eqnarray}\label{eq4.4*}
		|\mathcal{E}^*_\alpha(\tilde{F}_n)-	\mathcal{E}^*_\alpha({F})|&=&|\mathcal{E}^*_\alpha(\tilde{F}_n)-	\mathcal{E}^*_\alpha({F})+A_{\alpha}^{1}(\tilde{F}_n)-A_{\alpha}^{1}(\tilde{F}_n)+A_{\alpha}^{2}(F_n)-A_{\alpha}^{2}(F_{n})|\nonumber\\
		&\le& |\mathcal{E}^*_\alpha(\tilde{F}_n)-A_{\alpha}^{1}(\tilde{F}_n)|+|\mathcal{E}^*_\alpha(F)-A_{\alpha}^{2}(F_n)|+|A_{\alpha}^{1}(\tilde{F}_n)-A_{\alpha}^{2}(F_n)|\nonumber\\
		&=&T_1+T_2+T_3,~~ \text{(say)},
	\end{eqnarray}
	where the inequality follows from the well-known triangle inequality, and
	\begin{eqnarray*}
		A_{\alpha}^{1}(\tilde{F}_{n})&=&(-\alpha!)^{\frac{1}{\alpha}}\int_{0}^{\infty}  \tilde{F}_n(x)[\log \tilde{F}_n(x)]^\frac{1}{\alpha} dx,\\
		A_{\alpha}^{2}(F_{n})&=&(-\alpha!)^{\frac{1}{\alpha}}\int_{0}^{\infty} F_n(x)[\log F_n(x)]^\frac{1}{\alpha} dx.
	\end{eqnarray*}
Clearly, $T_{1}<\epsilon_1$ and $T_{2}<\epsilon_{2}$, since $\mathcal{E}^*_\alpha(\tilde{F}_n)\approx A_{\alpha}^{1}(\tilde{F}_n)$ and $\mathcal{E}^*_\alpha(F)\approx A_{\alpha}^{2}(F_n)$, respectively, where $\epsilon_{1}$ and $\epsilon_2$ are strictly positive small real numbers. Further, to show that $T_{3}<\epsilon_3$, where $\epsilon_3>0$, we consider 
	\begin{eqnarray*}
		\frac{\mathcal{E}^*_\alpha(\tilde{F}_n)}{(-1)^\frac{1}{\alpha}(\alpha!)^\frac{1}{\alpha}}&\approx&\int_{0}^{\infty}  \tilde{F}_n(x)[\log \tilde{F}_n(x)]^\frac{1}{\alpha} dx\\
		&=&\int_{0}^{1}  \tilde{F}_n(x)[\log \tilde{F}_n(x)]^\frac{1}{\alpha} dx+\int_{1}^{\infty}  \tilde{F}_n(x)[\log \tilde{F}_n(x)]^\frac{1}{\alpha} dx\\
		&=&\mathcal{I}_1+\mathcal{I}_2 ~~(say).
	\end{eqnarray*}
	Now, the rest of the proof follows using Theorem $14$ (taking $\varphi(x)=1$) of \cite{tahmasebi2020extension}. This completes the proof. 
\end{proof}

Next, we consider examples dealing with random samples from exponential and uniform distributions. 
	\begin{example}\label{ex4.1}
	Consider a random sample drawn from exponential distribution with parameter $\lambda$. From \cite{pyke1965spacings}, it is well-known that the sample spacings are independent and $U_i$ is exponentially distributed with parameter $\lambda(n-i)$. Thus, using (\ref{eq5.2}), we have
	\begin{eqnarray*}
		E[\mathcal{E}^*_\alpha(\tilde{F}_n)]\approx(\alpha!)^\frac{1}{\alpha}\sum_{i=1}^{n-1}\frac{1}{\lambda(n-i)}\bigg(\frac{i}{n}\bigg)\bigg(-\log \frac{i}{n}\bigg)^\frac{1}{\alpha}
	\end{eqnarray*}
	and 
	\begin{eqnarray*}
		Var[\mathcal{E}^*_\alpha(\tilde{F}_n)]\approx (\alpha!)^\frac{2}{\alpha}\sum_{i=1}^{n-1}\frac{1}{\lambda^2(n-i)^2}\bigg(\frac{i}{n}\bigg)^2\bigg(-\log \frac{i}{n}\bigg)^\frac{2}{\alpha}.
	\end{eqnarray*}
The values of expectation and variance of  $\mathcal{E}^*_\alpha(\tilde{F}_n)$ are presented in Table $4$ for different values of $\lambda$, $n$ and $\alpha.$ Here, we have considered $n=5,~10,~20,~50$, $\lambda=0.3,~0.7,~1.5$ and $\alpha=0.3,~0.4,~0.7,~0.9,~1.$ From the tabulated values, we observe that $E[\mathcal{E}^*_\alpha(\tilde{F}_n)]$ increases and $Var[\mathcal{E}^*_\alpha(\tilde{F}_n)]$ decreases with respect to $n$.  
\begin{table}[h!]
	\centering 
	\scalebox{.9}{\begin{tabular}{c c c c c c c c } 
			\hline\hline 
	\multirow{2}{.4cm}{$\lambda$} &	\multirow{2}{.4cm}{$n$} & $E[\mathcal{E}^*_{0.3}(\tilde{F}_{n})]$ &$E[\mathcal{E}^*_{0.4}(\tilde{F}_{n})]$ &$E[\mathcal{E}^*_{0.7}(\tilde{F}_{n})]$ &$E[\mathcal{E}^*_{0.9}(\tilde{F}_{n})]$
		&$E[\mathcal{E}^*_{1}(\tilde{F}_{n})]$ \\ [0.5ex] 
		&~&$	\big(Var[\mathcal{E}^*_{0.3}(\tilde{F}_n)]\big)$	&$	\big(Var[\mathcal{E}^*_{0.4}(\tilde{F}_n)]\big)$& $\big(Var[\mathcal{E}^*_{0.7}(\tilde{F}_n)]\big)$& $\big(Var[\mathcal{E}^*_{0.9}(\tilde{F}_n)]\big)$& $\big(Var[\mathcal{E}^*_1(\tilde{F}_n)]\big)$\\
			\hline\hline 
			\multirow{8}{1cm}{0.3} & 5 & 0.88 & 0.86 & 1.24& 1.59 & 1.78  \\
			~ & ~ &  (0.38)& (0.26) & (0.39) & (0.66) & (0.85)  \\[1.2ex]
			~ & 10 & 1.13 & 0.97 & 1.32 & 1.74 & 1.97  \\
			~ & ~ & (0.31) &(0.15) & (0.20) & (0.36) &(0.48)  \\[1.2ex]
			~ & 20 & 1.26 & 1.01 & 1.35 & 1.80 & 2.06 \\
			~ & ~ &(0.18) & (0.08) & (0.10) & (0.19) & (0.25)  \\[1.2ex]
			~ & 50 & 1.32 & 1.03 & 1.36 & 1.83 & 2.12  \\
			~ & ~ & (0.08) & (0.03) & (0.04) & (0.08) & (0.10) \\[1ex]
			\hline
			\multirow{8}{1cm}{0.7} & 5 & 0.38 & 0.37 & 0.53& 0.68 & 0.76  \\
			~ & ~ &  (0.07)& (0.05) & (0.07) & (0.12) & (0.16)  \\[1.2ex]
	    	~ & 10 & 0.48 & 0.41 & 0.57 & 0.75 & 0.85  \\
			~ & ~ & (0.06) &(0.03) & (0.04) & (0.07) &(0.09)  \\[1.2ex]
			~ & 20 & 0.54 & 0.43 & 0.58 &0.77 & 0.88 \\
			~ & ~ &(0.03) & (0.01) & (0.02) & (0.03) & (0.05)  \\[1.2ex]
			~ & 50 & 0.56 &  0.44 & 0.58 & 0.79  & 0.91 \\
			~ & ~ & (0.01) & (0.006) & (0.007) & (0.014) & (0.019) \\[1ex] 
				\hline
		\multirow{8}{1cm}{1.5} & 5 & 0.18 & 0.17 & 0.25& 0.32 & 0.36  \\
		~ & ~ &  (0.02)& (0.01) & (0.015) & (0.03) & (0.03)  \\[1.2ex]
		~ & 10 & 0.23 & 0.19 & 0.26 & 0.35 & 0.39  \\
		~ & ~ & (0.01) &(0.01) & (0.01) & (0.01) &(0.02)  \\[1.2ex]
		~ & 20 & 0.25 & 0.20 & 0.27 &0.36 & 0.41 \\
		~ & ~ &(0.007) & (0.003) & (0.004) & (0.007) & (0.01)  \\[1.2ex]
		~ & 50 & 0.26 & 0.21 & 0.27 & 0.37 & 0.42  \\
    	~ & ~ & (0.003) & (0.001) & (0.002) & (0.003) & (0.004) \\
			\hline\hline 
	\end{tabular}}
	\caption{Numerical values of $E[\mathcal{E}^*_{\alpha}(\tilde{F}_{n})]$ and $Var[\mathcal{E}^*_{\alpha}(\tilde{F}_n)]$ 
		for exponential distribution.}
	\label{tb1} 
\end{table}
\end{example}

 \begin{example}
	Let $(X_1,\ldots,X_n)$ be a random sample from the uniform distribution in the interval $[0,1]$. Then, the sample spacings are independent and follow beta distribution with parameters $1$ and $n$. For details see \cite{pyke1965spacings}. Then,
	\begin{eqnarray*}
		E[\mathcal{E}^*_\alpha(\tilde{F}_n)]\approx(\alpha!)^\frac{1}{\alpha}\sum_{i=1}^{n-1}\bigg(\frac{1}{n+1}\bigg)\bigg(\frac{i}{n}\bigg)\bigg(-\log \frac{i}{n}\bigg)^\frac{1}{\alpha}
	\end{eqnarray*}
	and 
	\begin{eqnarray*}
		Var[\mathcal{E}^*_\alpha(\tilde{F}_n)]\approx\frac{(\alpha!)^\frac{2}{\alpha}}{(n+1)^2(n+2)}\sum_{i=1}^{n-1}\bigg(\frac{i}{n}\bigg)^2\bigg(-\log \frac{i}{n}\bigg)^\frac{2}{\alpha}.
	\end{eqnarray*}
The values of expectation and variance of  $\mathcal{E}^*_\alpha(\tilde{F}_n)$ are presented in Table $4,$ for different values of $n$ and $\alpha.$ Similar behaviour for $E[\mathcal{E}^*_\alpha(\tilde{F}_n)]$ and $Var[\mathcal{E}^*_\alpha(\tilde{F}_n)]$ as in Example \ref{ex4.1}can be found from Table $5.$
\end{example}

\begin{table}[ht]
	\centering 
	\scalebox{.9}{\begin{tabular}{c c c c c c c c } 
			\hline\hline\vspace{.1cm} 
		\multirow{2}{.4cm}{$n$} & $E[\mathcal{E}^*_{0.3}(\tilde{F}_{n})]$ &$E[\mathcal{E}^*_{0.4}(\tilde{F}_{n})]$ &$E[\mathcal{E}^*_{0.7}(\tilde{F}_{n})]$ &$E[\mathcal{E}^*_{0.9}(\tilde{F}_{n})]$
		&$E[\mathcal{E}^*_{1}(\tilde{F}_{n})]$ \\ [0.5ex] 
		&$	\big(Var[\mathcal{E}^*_{0.3}(\tilde{F}_n)]\big)$	&$	\big(Var[\mathcal{E}^*_{0.4}(\tilde{F}_n)]\big)$& $\big(Var[\mathcal{E}^*_{0.7}(\tilde{F}_n)]\big)$& $\big(Var[\mathcal{E}^*_{0.9}(\tilde{F}_n)]\big)$& $\big(Var[\mathcal{E}^*_1(\tilde{F}_n)]\big)$\\
			\hline\hline 
			5 & 0.16 & 0.14 & 0.16& 0.18 & 0.20  \\
			 ~ &  (0.002)& (0.001) & (0.001) & (0.0012) & (0.0014)  \\[1.2ex]
			10 & 0.23 & 0.18 & 0.18 & 0.21 & 0.22  \\
			~ & (0.001) &(0.0006) & (0.0003) & (0.0005) &(0.0005)  \\[1.2ex]
			 20 & 0.28 & 0.20 & 0.19 & 0.22 & 0.24 \\
			 ~ &(0.0005) & (0.0002) & (0.0001) & (0.00013) & (0.00015)  \\[1.2ex]
			 35 & 0.298 & 0.21 & 0.20 & 0.23 & 0.24  \\
			 ~ & (0.00019) & (0.00007) & (0.00004) & (0.000048) & (0.00005) \\[1ex]
           50 & 0.306& 0.21 & 0.20 & 0.23 & 0.24  \\
          ~ & (0.00010) & (0.00003) & (0.00002) & (0.000024) & (0.000027) \\[1ex]
        100& 0.314 & 0.215 & 0.202 & 0.231 & 0.247  \\
       			 ~ & (0.00003) & ( $8.7\times 10^{-6}$) & ( $5.3\times 10^{-6}$) & ( $6.3\times 10^{-6}$) & ( $7.1\times 10^{-6}$) \\[1ex]
\hline\hline 
	\end{tabular}}
	\caption{Numerical values of $E[\mathcal{E}^*_{\alpha}(\tilde{F}_{n})]$ and $Var[\mathcal{E}^*_{\alpha}(\tilde{F}_n)]$ 
		for beta distribution.}
	\label{tb1} 
\end{table}

Now, we consider a random sample from exponential distribution and establish central limit theorem for the newly proposed measure. 
\begin{theorem}\label{th4.2}
	Suppose a random sample $(X_{1},\ldots,X_{n})$ is available from exponential distribution with mean $\theta>0.$ Then, for $0<\alpha<1,$ 
	$$\frac{\mathcal{E}^*_\alpha(\tilde{F}_n)-E[\mathcal{E}^*_\alpha(\tilde{F}_n)]}{\sqrt{Var [\mathcal{E}^*_\alpha(\tilde{F}_n)]}}\rightarrow N(0,1),~\mbox{in ~distribution~ as}~n\rightarrow\infty.$$
\end{theorem}
\begin{proof}
The proof follows along similar arguments as in the proof of  Proposition $5.2$ of \cite{di2021fractional}, and thus it is omitted. 
\end{proof}

\begin{remark}
	Using the result in Theorem \ref{th4.2}, the approximate confidence interval of the empirical EFCPE can be obtained as $$\left[\mathcal{E}^*_\alpha(\tilde{F}_n)-z_{\gamma/2}\sqrt{Var[\mathcal{E}^*_\alpha(\tilde{F}_n)]},\mathcal{E}^*_\alpha(\tilde{F}_n)+z_{\gamma/2}\sqrt{Var[\mathcal{E}^*_\alpha(\tilde{F}_n)]}\right],$$ where $z_{\gamma/2}$ is the upper $\gamma/2$ percentile of the standard normal distribution. 
\end{remark}

A quantity is said to be stable if the amount of its change under an arbitrary small deformation of the distribution remains small. Various authors have studied stability of different uncertainty measures. In this context, we refer to \cite{abe2002stability} and \cite{ubriaco2009entropies}. Along the similar research, \cite{xiong2019fractional} studied stability criteria for fractional cumulative residual entropy. Similar research has been carried out by \cite{di2021fractional} for fractional cumulative entropy. Very recently, \cite{foroghi2022extensions} considered stability of the empirical extended fractional cumulative residual entropy. Herein, we discuss stability of the empirical EFCPE. 
\begin{definition}
	Let $(X_1,\ldots,X_n)$ be a random sample from a population with cumulative distribution function $F(.)$ and $(X_{1}^*,\ldots,X_{n}^*)$ be any small deformation of it. Then, the extended EFCPE is stable if for all $\epsilon>0,$ there exists $\delta>0$ such that 
	$$\sum_{l=1}^{n}|X_{l}-X_{l}^*|<\delta\Rightarrow |\mathcal{E}^*_\alpha(\tilde{F}_n)-\mathcal{E}^*_\alpha(\tilde{F}_n^*)|<\epsilon,~\mbox{for}~n\in\mathbb{N}~\mbox{and}~0<\alpha<1.$$
\end{definition}
In the next result, we present sufficient condition, under which the EFCPE is stable. 
\begin{theorem}
	Let $X$ be a non-negative absolutely continuous random variable with distribution function $F(.)$. Then, the EFCPE is stable if $X$ has a distribution on a finite interval.  
\end{theorem}
\begin{proof}
	Let $X$ have a distribution on a finite interval.  Then, the extended EFCPE of $X$ is expressed as
	\begin{eqnarray}
	\mathcal{E}^*_\alpha(\tilde{F}_n)\approx (\alpha!)^{\frac{1}{\alpha}}\sum_{l=1}^{n-1}(X_{l+1:n}-X_{l:n}) \tilde{F}_{n}(X_{l:n})[-\log \tilde{F}_{n}(X_{l:n})]^{\frac{1}{\alpha}},~0<\alpha<1.
\end{eqnarray}
Now, the rest of the proof follows as in the proof of Theorem $5$ of \cite{xiong2019fractional}. Thus, it is omitted. 
\end{proof}

\begin{example}\label{ex4.3}
	In this example, we consider COVID-$19$ related weekly data set of size $20.$ The data set represents the number of deceased of people due to COVID in the state ODISHA of INDIA during $12$th April $2021$ to $23$rd August $2021.$ The information for the data set is available in the official website ``https://statedashboard.odisha.gov.in". The number of deceased due to COVID-$19$ virus per week are
	\begin{eqnarray*} 
	\{6,~20,~49,~76,~124,~138,~181,~238,~281,~311,~287,\\~297,~318,~414,~454,~458,~459,~468,~452,~473\}.
	\end{eqnarray*}
	The values of the empirical  extended fractional cumulative past entropy are computed based on this data set for some choices of $\alpha$, which are given by $\mathcal{E}^*_{0.2}(\tilde{F}_n)=424.411$, $\mathcal{E}^*_{0.4}(\tilde{F}_n)=123.741$, $\mathcal{E}^*_{0.8}(\tilde{F}_n)=125.559$ and $\mathcal{E}^*_1(\tilde{F}_n)=140.116.$ The values have been plotted in Figure $2(a)$ and Figure $2(b),$ for $0<\alpha<1$ and $0.2\le\alpha<1$, respectively.  The graphical plot in Figure $2(a)$ shows that $\mathcal{E}^*_{\alpha}(\tilde{F}_n)$ decreases with respect to $\alpha>0$ till $\alpha\approx0.55$, and increases afterwards. 
\end{example}

\begin{figure}[h]
	\begin{center}
		\subfigure[]{\label{c1}\includegraphics[height=1.8in]{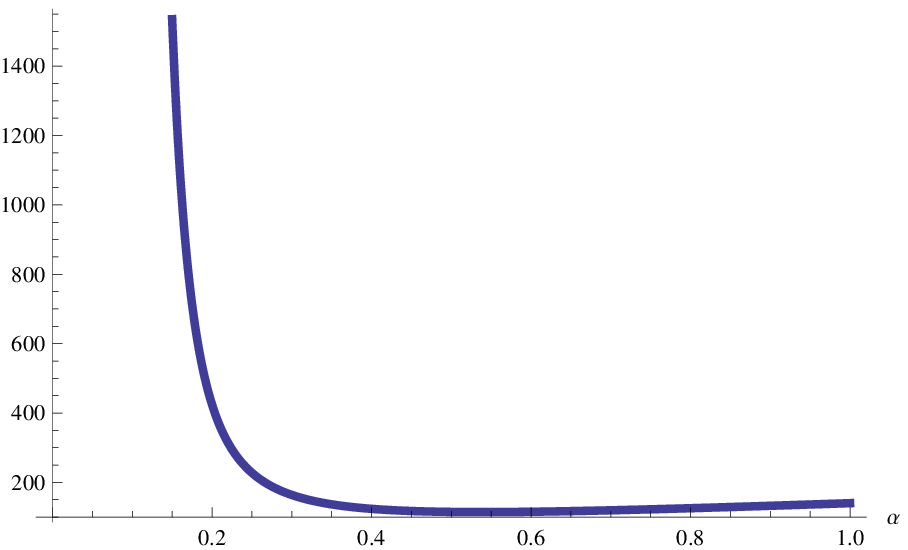}}
		\subfigure[]{\label{c1}\includegraphics[height=1.8in]{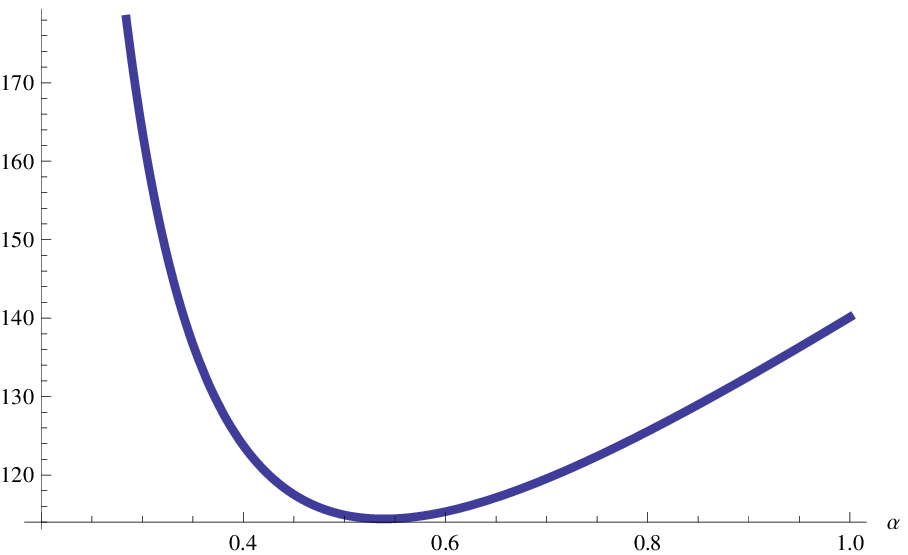}}
		\caption{$(a)$ Graph of the empirical EFCPE based on the real-life data set as in Example \ref{ex4.3} for $0<\alpha<1$.  $(b)$ Magnified view of the graph of the empirical EFCPE based on the real-life data set as in Example \ref{ex4.3} for $0.2\le\alpha<1.$}
	\end{center}
\end{figure}

\section{EFCPE of coherent system}
In this section, we study EFCPE for a coherent system. We recall that a system having $n$ number of components is called coherent if all the components are relevant, and if the system is monotone. Consider a coherent system with $n$ identically distributed (i.d) components. Denote by $T$ the lifetime of a coherent system. The CDF of $T$, denoted by $F_{T}(.)$ can be represented in terms of the distortion function $q(.)$ as (see \cite{navarro2014preservation})
\begin{eqnarray}\label{eq5.1*}
F_{T}(t)=q(F_{X}(t)),
\end{eqnarray} 
where $F_{X}(.)$ is a common CDF of the component lifetimes. Note that the distortion function with domain and codomain $[0,1]$ is continuous as well as increasing with $q(0)=0$ and $q(1)=1.$ In addition, the distortion function depends on the system structure and the copula associated with the component lifetimes. For a parallel system with independent and identically distributed $n$ components, $q(u)=u^n$ and for a $2$-out-of-$4$ system, $q(u)=6u^4-8u^3+3u^2.$ Define $\phi_{\alpha}(u)=u[-Ln_{\alpha}u]^{1/\alpha},~0<u<1$. Then,  using (\ref{eq5.1*}), the EFCPE of $T$ can be written as
\begin{eqnarray}\label{eq5.2*}
\mathcal{E}^*_{\alpha}(T)=\int_{0}^{\infty}\phi_{\alpha}(F_{T}(x))dx
=\int_{0}^{\infty}\phi_{\alpha}(q(F_{X}(x)))dx=\int_{0}^{1}\frac{\phi_{\alpha}(q(u))}{f_{X}(F_{X}^{-1}(u))}du,~~0<\alpha<1,
\end{eqnarray}
where the final equality is obtained using the transformation $u=F_{X}(x).$ Next, we consider a coherent system with lifetime $T=\max\{X_1,X_2\}$, where $X_1$ and $X_2$ are component lifetimes. It is assumed that $X_1$ and $X_2$ are independent and follow uniform distribution in the interval $(0,1).$ Thus, from (\ref{eq5.2*}), we have 
\begin{eqnarray}\label{eq5.3}
\mathcal{E}^*_{\alpha}(T)\approx (\alpha!)^{\frac{1}{\alpha}}\int_{0}^{1} u^2 [-\log u^2]^{\frac{1}{\alpha}}du=\frac{(2\alpha!)^{1/\alpha}\Gamma(\frac{1}{\alpha}+1)}{3^{{1/\alpha}+1}}.
\end{eqnarray}
\begin{figure}[h]
	\begin{center}
		\includegraphics[height=1.8in]{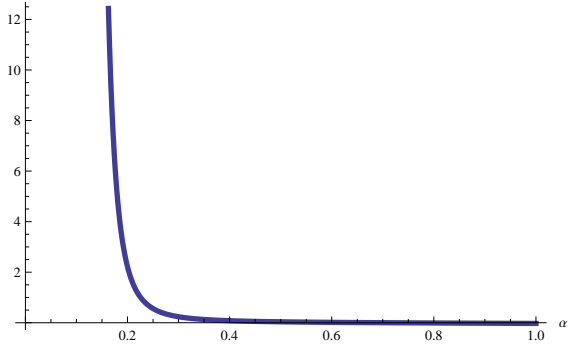}
		\caption{ Graph of the difference of  $\mathcal{E}^*_{\alpha}(T)- \mathcal{E}^*_{\alpha}(X_1),$  for $\alpha\in(0,1)$.}
	\end{center}
\end{figure}
We have plotted the difference $\mathcal{E}^*_{\alpha}(T)- \mathcal{E}^*_{\alpha}(X_1),$  for $\alpha\in(0,1)$ in Figure $3$, which shows that $\mathcal{E}^*_{\alpha}(T)\ge \mathcal{E}^*_{\alpha}(X_1),$ that is, uncertainty of a coherent system is larger than that of its components. Thus, question arises: whether one can generalize this statement for a general system. The following proposition provides answer of it under a condition.
\begin{proposition}
	Suppose $T$ denotes the lifetime of a coherent system with identically distributed components. The distortion function is denoted by $q(.)$. If $\phi_\alpha(u)=u[-Ln_\alpha u]^{1/\alpha}$ and  $\phi_\alpha\big(q(u)\big)\geq (resp.\leq)\phi_\alpha(u),$ for $0<\alpha, u<1$, then $\mathcal{E}^*_\alpha(T)\geq (resp. \leq)\mathcal{E}^*_\alpha(X)$.
\end{proposition}	 	
\begin{proof}
The proof is straightforward, and thus it is omitted.
\end{proof}

Another interesting result associated with the comparison of EFCPE of two systems when they have same structure and different i.d component lifetimes. 
\begin{proposition}
	Suppose $T_1$ and $T_2$ are the lifetimes of two different coherent systems with same structure and respective i.d component lifetimes $X_1,\ldots,X_n$ and $Y_1,\ldots,Y_n$ with same copula. The common CDFs for $X_1,\ldots,X_n$ and $Y_1,\ldots,Y_n$ are denoted by $F_{X}(.)$ and $F_{Y}(.)$, respectively. 
	\begin{itemize}
		\item [(i)] If $F_X\leq^{disp}F_Y$, then $\mathcal{E}^*_\alpha(T_1)\leq\mathcal{E}^*_\alpha(T_2),~0<\alpha<1.$
		\item[(ii)] If $ \mathcal{E}^*_\alpha(X)\leq\mathcal{E}^*_\alpha(Y)$ and $\inf_{u\in\beta_1}\frac{\phi_\alpha(q(u))}{\phi_\alpha(u)}\geq \sup_{u\in \beta_2}\frac{\phi_\alpha(q(u))}{\phi_\alpha(u)}$, 
		for $\beta_1=\{u\in[0,1] :f_X\big(F^{-1}_X(u)\big)>f_Y\big(F^{-1}_Y(u)\big)\}$
		and $\beta_2=\{u\in[0,1] :f_X\big(F^{-1}_X(u)\big)\leq f_Y\big(F^{-1}_Y(u)\big)\},$
		then $\mathcal{E}^*_\alpha(T_1)\leq\mathcal{E}^*_\alpha(T_2)$, $0<\alpha<1.$
	\end{itemize}
\end{proposition}
\begin{proof}
	$(i)$ Both systems have a common distortion function $q$, since the systems have same structure and a common copula. Further, the assumption  $F_X\leq^{disp}F_Y$ implies $f_X\big(F^{-1}_X(u)\big)\geq f_Y\big(F^{-1}_Y(u)\big)~~ \forall~ u\in(0,1).$ Thus,  $$\frac{\phi_\alpha(q(u))}{f_X\big(F^{-1}_X(u)\big)}\leq \frac{\phi_\alpha(q(u))}{f_Y\big(F^{-1}_Y(u)\big)}.$$
	Hence using (\ref{eq5.2*}), the result readily follows.
\\
$(ii)$ We have $\mathcal{E}^*_\alpha(X)\leq\mathcal{E}^*_\alpha(Y)$ implies  $\int_{0}^{1}\Delta(u)du\ge0$, where $\Delta(u)=\frac{\phi_\alpha(u)}{f_Y\big(F^{-1}_Y(u)\big)}-\frac{\phi_\alpha(u)}{f_X\big(F^{-1}_X(u)\big)}$.
Now, 
\begin{eqnarray*}
	\mathcal{E}^*_\alpha(T_2)-\mathcal{E}^*_\alpha(T_1)&=&\int^1_0\bigg(\frac{\phi_\alpha\big(q(u)\big)}{f_Y\big(F^{-1}_Y(u)\big)}-\frac{\phi_\alpha\big(q(u)\big)}{f_X\big(F^{-1}_X(u)\big)}\bigg)du\\
	&=&\int^1_0\frac{\phi_\alpha\big(q(u)\big)}{\phi_\alpha(u)}\Delta(u)du\\
	&=&\int_{\beta_1}\frac{\phi_\alpha\big(q(u)\big)}{\phi_\alpha(u)}\Delta(u)du+\int_{\beta_2}\frac{\phi_\alpha\big(q(u)\big)}{\phi_\alpha(u)}\Delta(u)du\\
	&\geq&\inf_{u\in \beta_1}\frac{\phi_\alpha\big(q(u)\big)}{\phi_\alpha(u)}\int^1_{0}\Delta(u)du+\sup_{u\in \beta_2}\frac{\phi_\alpha\big(q(u)\big)}{\phi_\alpha(u)}\int^1_{0}\Delta(u)du\\
	&\geq&\bigg(\sup_{u\in \beta_2}\frac{\phi_\alpha\big(q(u)\big)}{\phi_\alpha(u)}\bigg)\int^1_{0}\Delta(u)du\geq0.
\end{eqnarray*}
Thus, the proof is completed.
\end{proof}

Next, we obtain bounds of the EFCPE of a coherent system. There are many cases, where exact value of the EFCPE of a system can not be obtained. The reason might be due to the complicated system structure or the number of components is large. In this case, bounds are important to study various characteristics of the coherent system. In the following result, the bounds of the EFCPE of system lifetime are obtained in terms of that of the component lifetime. 

\begin{proposition}\label{prop5.3}
	Let $T$ be the lifetime of a coherent system with i.d components and its distortion function be $q$, and $\phi_\alpha(u)=u\big[-Ln_\alpha(u)\big]^\frac{1}{\alpha}$. Then, we have obtained 
	\begin{eqnarray*}
		\omega_{1,\alpha}\mathcal{E}^*_\alpha(X)\leq\mathcal{E}^*_\alpha(T)\leq\omega_{2,\alpha}\mathcal{E}^*_\alpha(X) ~~\forall~ 0<\alpha<1,
	\end{eqnarray*}
	where $\omega_{1,\alpha} =\displaystyle\inf_{u\in(0,1)}\frac{\phi_\alpha(q(u))}{\phi_\alpha(u)}$ and  $\omega_{2,\alpha} =\displaystyle\sup_{u\in(0,1)}\frac{\phi_\alpha(q(u))}{\phi_\alpha(u)}$.
\end{proposition}	 	

\begin{proof}
	By utilizing  (\ref{eq5.2*}), we get
	\begin{eqnarray*}
		\mathcal{E}^*_\alpha(T)
		&=&\int^1_0\frac{\phi_\alpha\big(q(u)\big)}{\phi_\alpha(u)}\times \frac{\phi_\alpha(u)}{f_X\big(F^{-1}_X(u)\big)}du\\
		&\leq& \sup_{u\in(0,1)}\frac{\phi_\alpha\big(q(u)\big)}{\phi_\alpha(u)}\int^1_0\frac{\phi_\alpha(u)}{f_X\big(F^{-1}_X(u)\big)}du\\
		&=&\bigg(\sup_{u\in(0,1)}\frac{\phi_\alpha\big(q(u)\big)}{\phi_\alpha(u)}\bigg)\mathcal{E}^*_\alpha(X)
		=\omega_{2,\alpha}\mathcal{E}^*_\alpha(X).
	\end{eqnarray*}
	Analogously, we obtain $\mathcal{E}^*_\alpha(T)\geq\omega_{1,\alpha}\mathcal{E}^*_\alpha(X).$
	Thus, the result follows.
\end{proof}

\begin{example}\label{ex5.1}
 Let us consider a coherent system with lifetime $T=\max\{X_1,X_2\}$, where $X_1$ and $X_2$ are component lifetimes. It is assumed that $X_1$ and $X_2$ are independent and follow uniform distribution in the interval $(0,1)$. The values of $\omega_{1,\alpha}$, $\omega_{2,\alpha},$ $\mathcal{E}^*_\alpha(X)$ (see Example \ref{ex2.1}(i)), $\omega_{2,\alpha}\mathcal{E}^*_\alpha(X),$ and $\mathcal{E}^*_\alpha(T)$ (see (\ref{eq5.3})) are presented in Table $6.$

 \begin{table}[h!]
 	\centering
 	\begin{tabular}{ | m{2em} | m{1cm}| m{2cm} | m{2cm} | m{2cm} | m{2cm} |} 
 		\hline\vspace{.1cm}
 		$\alpha$ & $\omega_{1,\alpha}$ & $\omega_{2,\alpha}$& $\mathcal{E}^*_\alpha(X)$ &$\omega_{2,\alpha}\mathcal{E}^*_\alpha(X)$& $\mathcal{E}^*_\alpha(T)$\\ 
 		\hline\vspace{.1cm}
 		0.1 & 0 & 20.5656 & 1076.068 & 22129.9893 &12739.0173\\ 
 		\hline\vspace{.1cm}
 		 		0.3 & 0 & 10.0784 & 0.3203 & 3.2282 & 0.5571\\ 
 		 \hline\vspace{.1cm}
 		  		 		0.5 & 0 & 3.9996 & 0.19635 & 0.7853 & 0.2327\\
 			\hline\vspace{.1cm}
 		  		    	0.7 & 0 & 2.6915 & 0.2050 & 0.5519 & 0.2062\\		
 		\hline
 	
 	\end{tabular}
 \caption{Approximate values for $\mathcal{E}^*_\alpha(T)$ and  the bounds of $\mathcal{E}^*_\alpha(T)$, for some specific values of $\alpha$ for Example \ref{ex5.1}.}
 \end{table}
\end{example}

Below, we obtain a result similar to the preceding proposition. It also compares two coherent systems. 
\begin{proposition}
	Let $T_1$ and $T_2$ be the lifetimes of two coherent systems with i.d. components and distortion functions $q_1$ and $q_2,$ respectively and $\phi_\alpha(u)=u\big[-Ln_\alpha u\big]^\frac{1}{\alpha}$. Then, 
	$$\bigg(\inf_{u\in (0,1)}\frac{\phi_\alpha\big(q_2(u)\big)}{\phi_\alpha\big(q_1(u))}\bigg)\mathcal{E}^*_\alpha(T_1)\leq \mathcal{E}^*_\alpha(T_2)\leq \bigg(\sup_{u\in (0,1)}\frac{\phi_\alpha\big(q_2(u)\big)}{\phi_\alpha\big(q_1(u))}\bigg)\mathcal{E}^*_\alpha(T_1)~~\forall~0<\alpha<1.$$
\end{proposition}
\begin{proof}
	The proof is analogous to that of Proposition \ref{prop5.3}, and thus it is omitted. 
\end{proof}

There are some distributions which have bounded density functions. The following result provides bounds of the EFCPE of a system lifetime when the i.d components have bounded density function.
\begin{proposition}\label{prop5.5}
	Let $T$ be the lifetime of a coherent system with i.d components and distortion function $q(.)$. Further, let $\phi_\alpha(u)=u\big[-Ln_\alpha u\big]^\frac{1}{\alpha}$ and the component lifetimes have absolutely continuous CDF $F_X(.)$ and PDF $f_X(.)$ with support $S$. 
	\begin{itemize}
		\item[(i)] If $f_X(x)\leq M ~~ \forall~ x\in S$, then
		$\mathcal{E}^*_\alpha(T)\geq \frac{1}{M}\int^1_0\phi_\alpha\big(q(u)\big)du~~\forall~0<\alpha<1.$
		\item[(ii)]If $f_X(x)\geq L>0 ~~ \forall~ x\in S$, then
		$\mathcal{E}^*_\alpha(T)\leq \frac{1}{L}\int^1_0\phi_\alpha\big(q(u)\big)du~~\forall~0<\alpha<1.$
	\end{itemize} 
\end{proposition}
\begin{proof}
	$(i)$ Using (\ref{eq5.2*}), we have 
	\begin{eqnarray*}
		\mathcal{E}^*_\alpha(T)&=&\int^1_0\frac{\phi_\alpha\big(q(u)\big)}{f_X\big(F^{-1}_X(u)\big)}du
		\geq \frac{1}{M}\int^1_0\phi_\alpha\big(q(u)\big)du~~~~(as~ f_X\big(F^{-1}_X(u)\big)\leq M).
	\end{eqnarray*}
	Hence the result follows.\\
	$(ii)$  The proof of this part is similar to that of the first part. Thus, it is omitted.
\end{proof}

The following example illustrates Proposition \ref{prop5.5}.

 \begin{example}
	$(i)$ If $T$ is the lifetime of a coherent system with i.d components with CDF $F_X(x)=1-\big(\frac{\beta}{\beta+x}\big)^\gamma, ~x>0$ and $\gamma,\beta>0$, then $M=\gamma/\beta$ and from Proposition 
	\ref{prop5.5}$(i)$
	$$\mathcal{E}^*_\alpha(T)\geq \frac{\beta}{\gamma}\int^1_0\phi_\alpha\big(q(u)\big)du,~ 0<\alpha<1.$$
	$(ii)$ If $T$ denotes the lifetime of a coherent system with i.d components following log-uniform distribution with CDF $F_{X}(x)=\frac{\log x-\log a}{\log b-\log a},~0<a\le x \le b$, then $L=\frac{1}{b\log(b/a)}$ and as a result from Proposition 
	\ref{prop5.5}$(ii)$, we have 
	$$\mathcal{E}^*_\alpha(T)\leq b\log\left(\frac{b}{a}\right)\int^1_0\phi_\alpha\big(q(u)\big)du,~ 0<\alpha<1.$$
\end{example}

\section{Validation and application of EFCPE}
In this section, we consider logistic map and test the acceptability of the newly proposed uncertainty measure. The logistic map is given by
\begin{eqnarray}
x_{l+1}=s x_l(1-x_l),
\end{eqnarray}
where $x\in[0,1]$, $s\in[0,4]$ and $l=0,1,2,\ldots$. Note that the logistic map is applied for the study of chaotic behaviour of a system. With the changes in $s$, the simulated data using logistic map have different characteristics such as periodical series and chaotic states. Here, the data have been generated by taking initial value $x_0=0.1$ for different choices of $s$ such as $3.58,~3.6,~3.7,~3.8$ and $4$. The bifurcation diagram of the logistic map is presented in Figure $4(a)$. Corresponding plots of empirical EFCPE given by (\ref{eq5.2}) of the generated series of data are presented in Figure $4(b)$, from which it is clearly visible that the newly proposed uncertainty measure can perfectly capture the difference of uncertainty between chaotic and periodic series. As expected, from  Figure $4(b),$ we observe that the EFCPE provides higher entropy values than periodic ones for $s=3.8$ and $4.$ Further, as $\alpha$ increases, the degree of randomness increases. It is highest for all values of $\alpha$ when the logistic map is fully chaotic, that is when $s=4.$ These observations demonstrate that the EFCPE is a valid measure of uncertainty.

\begin{figure}[h]
	\begin{center}
		\subfigure[]{\label{c1}\includegraphics[height=1.95in]{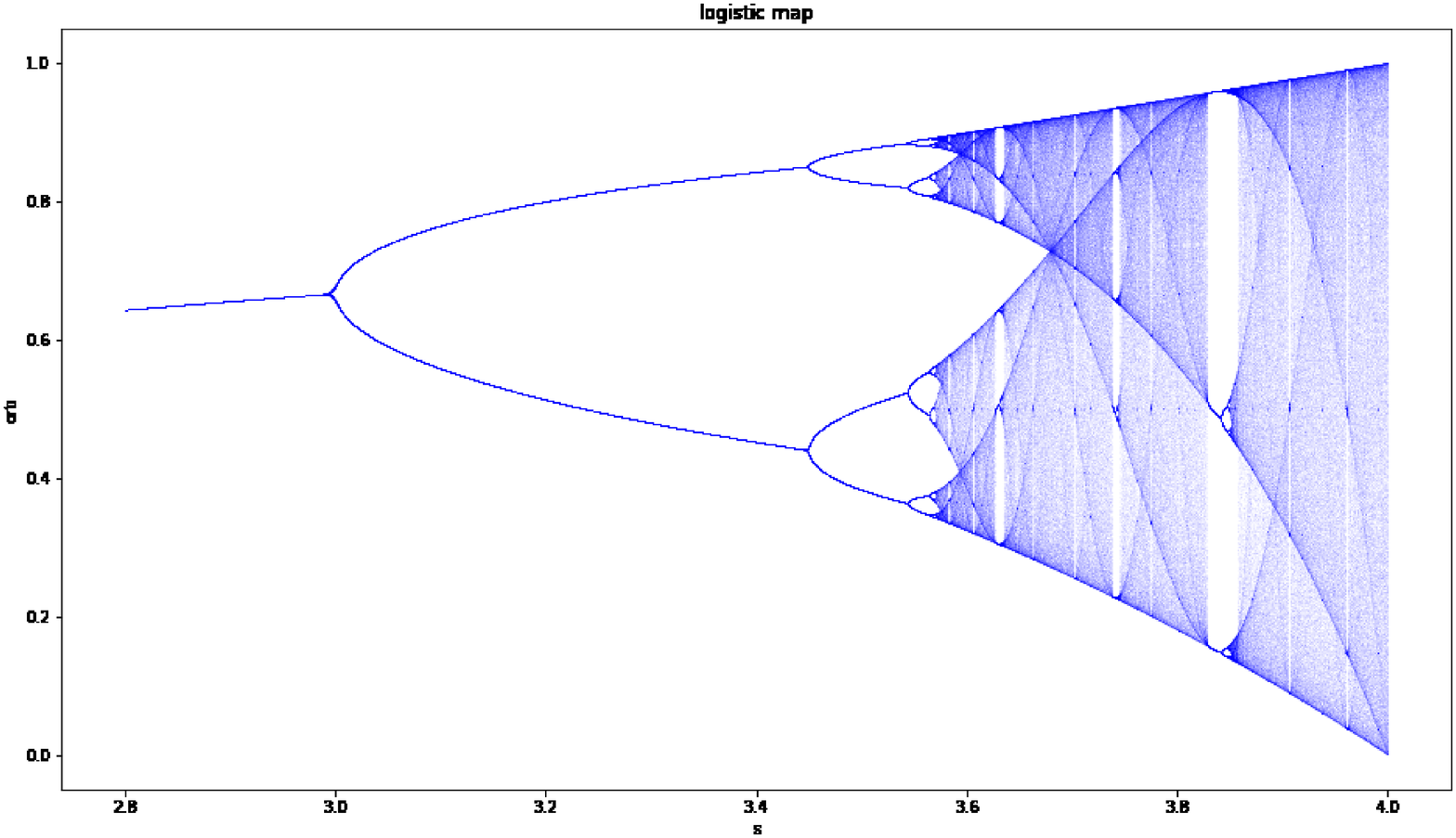}}
		\subfigure[]{\label{c1}\includegraphics[height=1.95in]{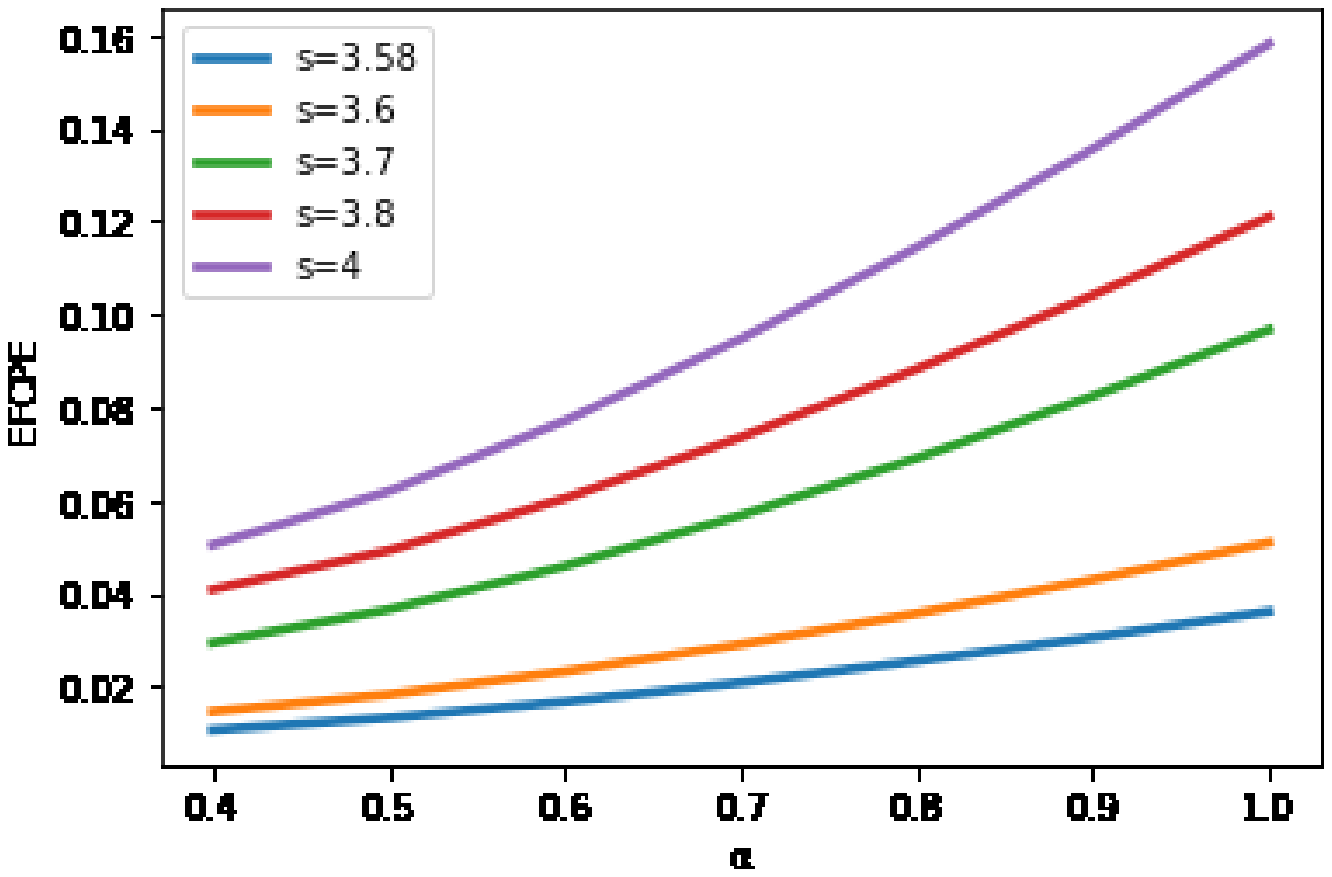}}
		\caption{$(a)$ Bifurcation diagram of the logistic map.  $(b)$ Graphs of EFCPE for different choices of $s$ based on the data generated using logistic map. Here, we have considered $s=3.58,~3.6,~3.7,~3.8,~4$ (from bottom to top).}
	\end{center}
\end{figure}

 In order to see the applicability of the proposed measure, here, we present a comparison study between EFCPE in (\ref{eq2.2}) with cumulative entropy (CE) proposed by \cite{di2009cumulative}  for the Weibull distribution with scale parameter $1$ and shape parameter $5$. Figure $5(a)$ depicts the plots of EFCPE for $\alpha= 0.15~,0.20~,0.23$ and the CE, $5(b)$ presents the plots of EFCPE for $\alpha= 0.30,~0.35,~0.40$ and the CE and $5(c)$ provides the plots of EFCPE for $\alpha= 0.97,~0.98,~0.99$ and the CE. 
 
  Note that the entropies are equal to the areas below each curve. From Figure $5$, it is easy to check that the areas under each curve for the EFCPE are larger than the area under the CE. As expected, Figure $5(c)$ shows that all the curves are very close to each other when $\alpha$ is near to $1$. Thus, one can conclude using these graphs that the newly proposed fractional measure is better one than the CE for $0<\alpha<1$. 
\begin{figure}[h]
	\begin{center}
		\subfigure[]{\label{c1}\includegraphics[height=1.92in]{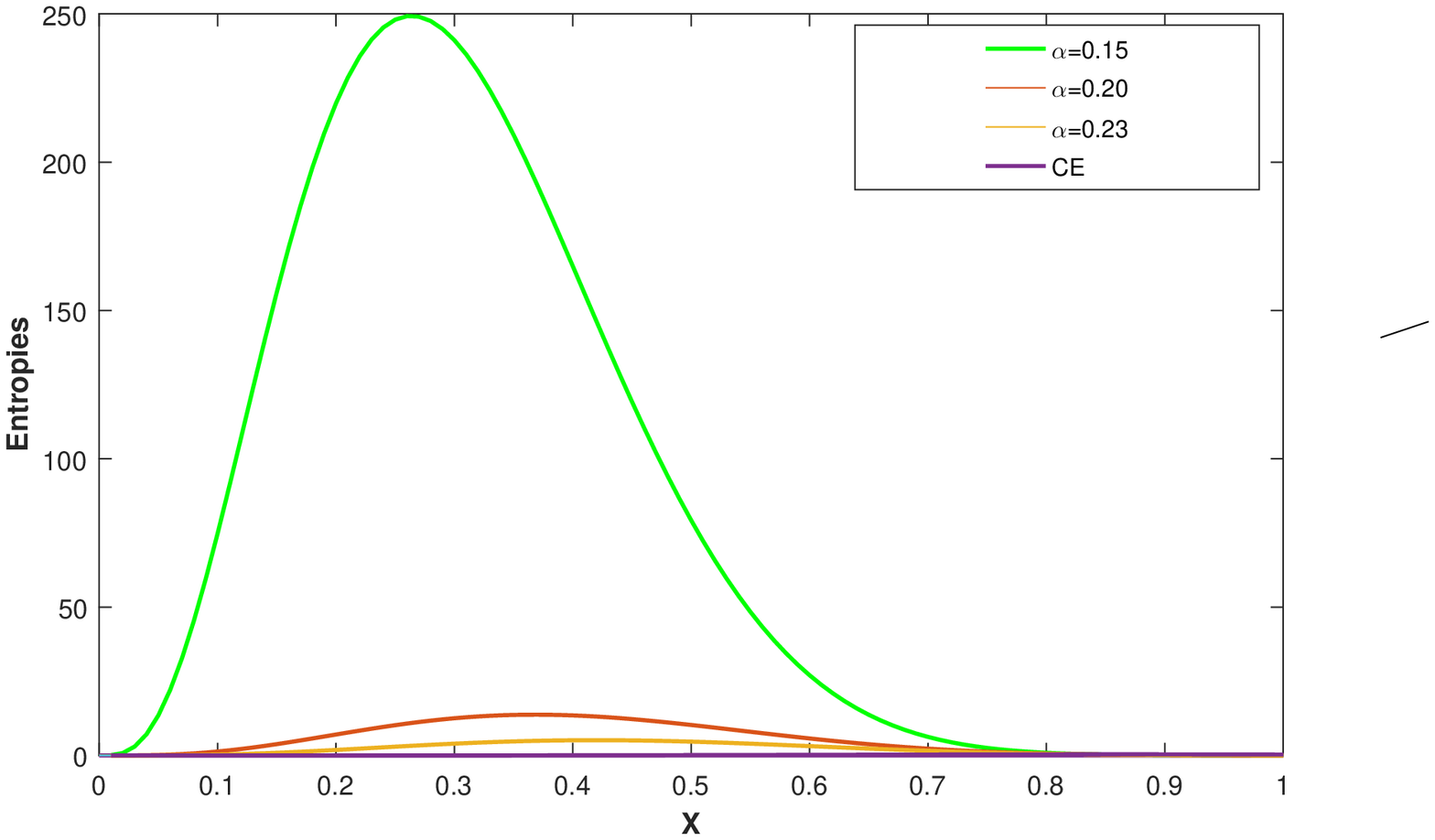}}
		\subfigure[]{\label{c1}\includegraphics[height=1.92in]{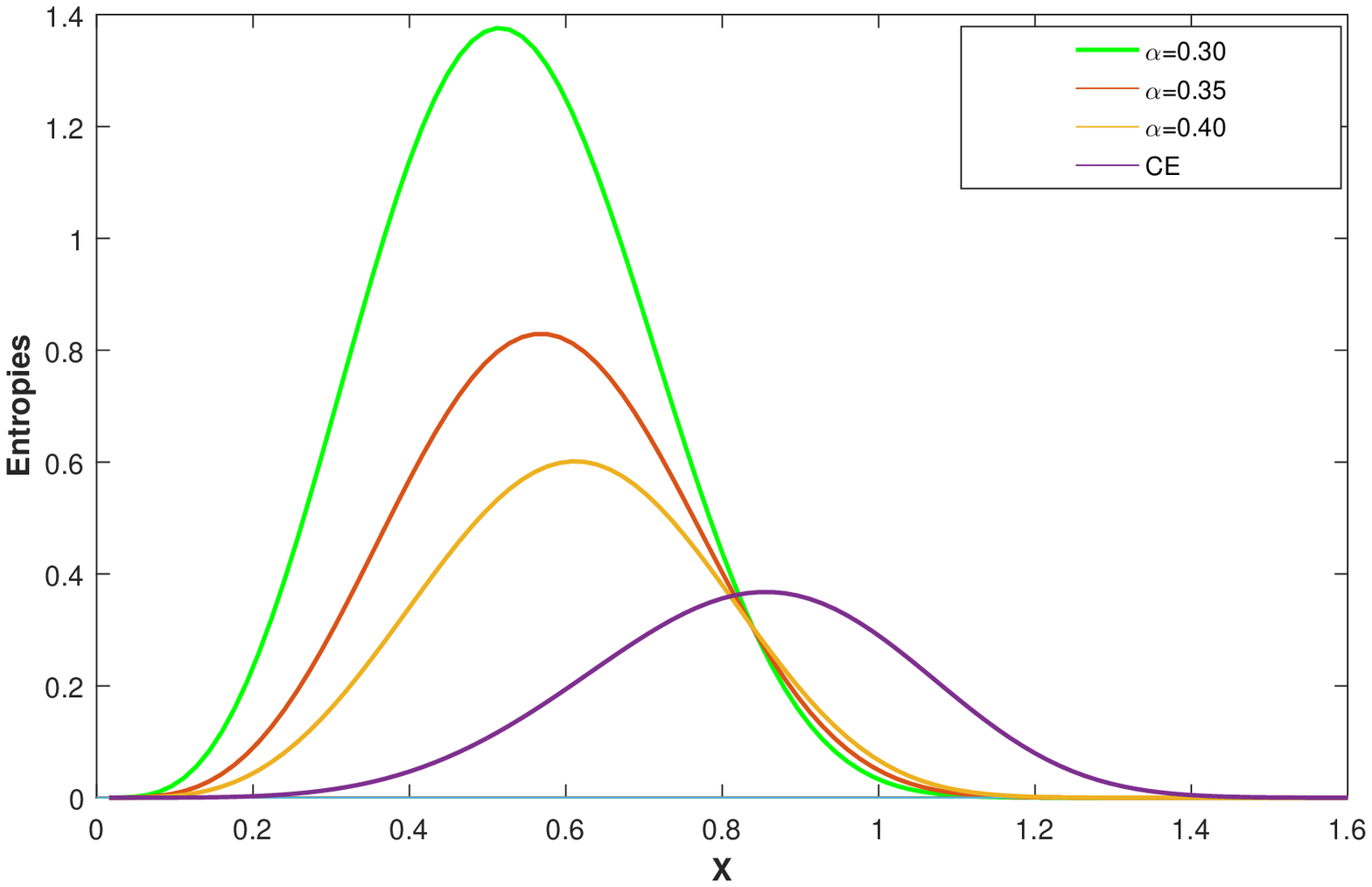}}
		\subfigure[]{\label{c1}\includegraphics[height=1.92in]{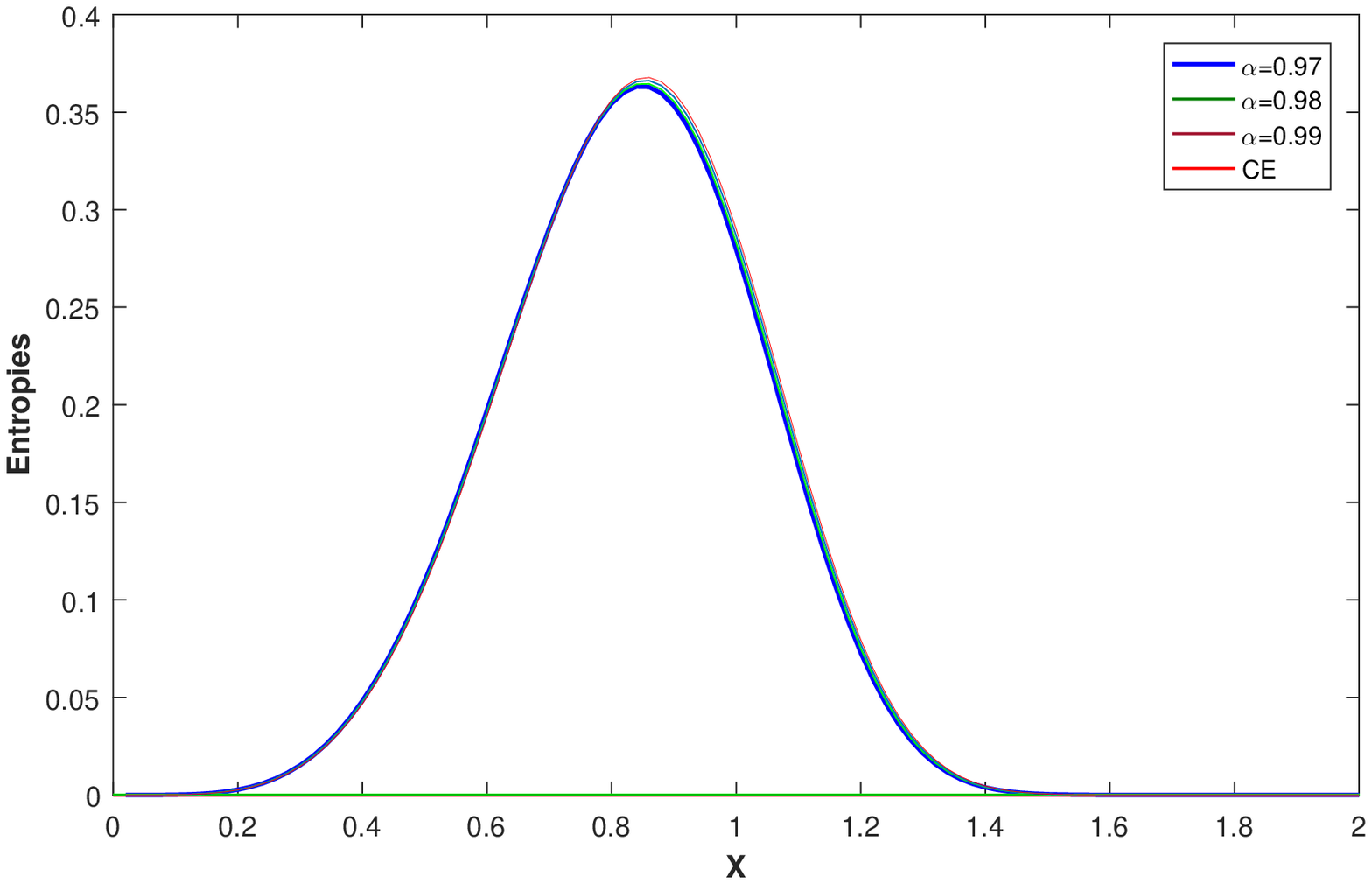}}
		\caption{A comparative study of the EFCPE and CE for $(a)$ $\alpha= 0.15~,0.20~,0.23$;  $(b)$ $\alpha= 0.30,~0.35,~0.40$ and $(c)$ $\alpha= 0.97,~0.98,~0.99$.}
	\end{center}
\end{figure}

\section{Conclusions}
Entropy measures the uncertainty or heterogeneity in a physical system. In this paper, we have proposed a new entropy, known as EFCPE and explored several properties of it. This concept is illustrated for the bivariate setup. Bounds are obtained. A connection between the stochastic order and larger uncertainty in terms of the EFCPE is established. Further, the concept of conditional EFCPE has been explored. The newly proposed measure is studied for the past lifetime. In addition, we have proposed another new concept for extended fractional cumulative paired $\phi$-entropy. The empirical EFCPE is proposed for the purpose of estimation. The empirical estimator is illustrated using two examples associated with exponential and uniform distributions. The stability of the proposed measure is studied. A COVID-$19$ related data set is considered to compute the values of EFCPE for different choices of $\alpha$. Further, various properties of the EFCPE of coherent systems are proposed. Finally, validation and application of EFCPE are provided. The validation has been explained using a simulation study on logistic map. It is observed that more chaos in a system leads to higher uncertainty. For the application of the proposed measure, we considered Weibull distribution and checked that the newly proposed fractional uncertainty measure (EFCPE) provides better output than the cumulative entropy measure proposed by \cite{di2009cumulative}. 

\section*{Acknowledgements} Both the authors thank the referees for helpful comments and suggestions which lead to the substantial improvements. The author S. Saha thanks the UGC (Award No. $191620139416$), India, for the financial assistantship received to carry out this research work. Both the authors are thankful to Ritik Roshan Giri for plotting the graphs in Figure $4$. 

\section*{Data availability statement} Publicly available data were analyzed in this study. This can be found in the website ``https://statedashboard.odisha.gov.in".

\section*{Conflicts of interest} The authors declare no conflict of interest. 
	   
	\bibliography{refference}
		 \end{document}